\documentclass[12pt]{amsart}

\usepackage{amsthm, amsmath, amssymb, amsrefs, color}

\usepackage{thmtools, thm-restate}
\declaretheorem[numberwithin=section]{theorem}

\usepackage{fullpage}

\usepackage{hyperref}

\newcommand{\T}{\mathbb{T}} 
\newcommand{\C}{\mathbb{C}} 
\newcommand{\D}{\mathbb{D}} 
\newcommand{\ints}{\mathbb{Z}} 
\newcommand{\uhp}{\mathbb{H}}
\newcommand{\R}{\mathbb{R}}

\newcommand{\p}{\mathfrak{p}}
\newcommand{\mcI}{\mathcal{I}}
\renewcommand{\Im}{\mathrm{Im}}
\newcommand{\refl}[1]{\bar{#1}}

\numberwithin{equation}{section}

\newif\ifletter

\letterfalse

\newtheorem{prop}[theorem]{Proposition}

\newtheorem{lemma}[theorem]{Lemma}

\theoremstyle{definition}

\newtheorem{remark}[theorem]{Remark}
\newtheorem{example}[theorem]{Example}


\address{Washington University in St. Louis\\ Department of Mathematics\\ One Brookings Drive\\ St. Louis, MO 63130 USA}

\email{geknese@wustl.edu}

\title{Boundary local Integrability of Rational Functions in two variables}

\author{Greg Knese}

\date{\today}

\thanks{Partially supported by NSF grants DMS-1900816 and DMS-2247702}

\begin{document}

\maketitle

\begin{abstract}
Motivated by studying boundary singularities of rational
functions in two variables that are analytic
on a domain, we investigate local integrability
on $\R^2$ near $(0,0)$
of rational functions with denominator non-vanishing
in the bi-upper half-plane but with an isolated
zero (with respect to $\R^2$) at the origin.
Building on work of Bickel-Pascoe-Sola \cites{BPS18,BPS20},
we give a necessary and sufficient test for membership
in a local $L^{\p}(\R^2)$ space and we give a complete description
of all numerators $Q$ such that $Q/P$ is locally
in a given $L^{\p}$ space.  
As applications, we prove that every bounded
rational function on the bidisk has partial derivatives
belonging to $L^1$ on the two-torus.
In addition, we give a new proof of a conjecture started
in \cite{BKPS} and completed by Koll\'ar \cite{kollar} 
characterizing the ideal of $Q$ such
that $Q/P$ is locally bounded.
A larger takeaway from this work is that a local model
for stable polynomials we employ is a flexible tool and may be
of use for other local questions about stable polynomials.  
\end{abstract}

\tableofcontents

\section{Introduction}

This paper continues the study of boundary singularities of rational
functions in several complex variables.  
Our starting point is a rational function $Q(z,w)/P(z,w)$
where $P$ has no zeros in the bidisk $\D^2 = \{(z,w) \in \C^2: |z|,|w|<1\}$ yet $P$ may have zeros on the boundary, say
$P(1,1)=0$.  
There are many ways to study the boundary
singularity and how its nature affects
the related rational function.
One way to understand the singularity at $(1,1)$
is to determine the integrability of $Q/P$ on 
the two-torus $\T^2 = \{(z,w)\in \C^2: |z|=|w|=1\}$.  
In particular, we can ask if $Q/P$ belongs to $L^2(\T^2)$ or
some other $L^{\p}$ space.  The question of when
$Q/P$ is bounded was addressed in the works \cites{BKPS, kollar}, but 
we give a new approach to this aspect here as well.

Our interest in this question originated
from the study of stable polynomials.
Stable polynomials are polynomials
with no zeros on either $\D^d$ or the product of upper
half planes $\uhp^d$.
They appear in many different areas
of mathematics  (see the introduction of \cite{BKPS} for references).
One could of course study polynomials with no zeros
on some other domain in $\C^d$, 
however the reflection across the circle or real line
makes polynomials with no zeros on $\D^d$ or $\uhp^d$
especially tractable.  Applications to combinatorics and 
other areas have most likely emphasized product domains because
of a necessary independence of the variables involved.
(That said, it would certainly be interesting to
develop a theory for other domains---the method of using Puiseux series
in this paper and \cite{BKPS} offers some promise for domains in $\C^2$.)
There is an interesting connection between
stable polynomials and sums of squares formulas associated to them
where a natural issue of membership in $L^2$ arises; \cites{gKpnozb, gKintreg}.  
We discuss this connection in greater detail in Section \ref{sec:sos}.  
Understanding integrability and boundedness of rational functions
on $\D^2$ then naturally leads to the study of
local properties of stable polynomials around a boundary zero.
The paper \cite{BKPS} presented a detailed local description 
near a boundary zero of stable polynomials in two variables.
One outcome of that work is that for some 
questions our initial stable polynomial $P$ can be replaced
by a simpler polynomial $[P]$ with a property weaker than being stable; 
see Theorem \ref{brackthm} below.
One big takeaway from the \emph{present} work is that this local model
is surprisingly robust.
Our specific interpretation of robust here is that much of the local structure of
the real and imaginary parts of $P$ surprisingly carries over to the real and imaginary 
parts of $[P]$.
At a more practical level, we learn about membership
in some classical function spaces.
Membership of an analytic function in one of the boundary
$L^{\p}$ spaces is really a question about membership 
in a Hardy space $H^{\p}(\D^d)$.  For $\p \in (0,\infty]$, the Hardy space
on the polydisk $H^{\p}(\D^d)$ consists of analytic functions $f:\D^d\to \C$
on the polydisk satisfying
\[
\|f\|_{H^{\p}} := \sup_{0<r<1} \|f_r\|_{L^{\p}(\T^d)} <\infty
\]
where $f_r(z):= f(rz)$.
The Hardy spaces on the disk and polydisk have 
been studied for decades (see the books \cite{Rudin} and \cite{Dirichletbook}) 
and here we try to 
simply understand which rational functions belong
to these spaces---this seems to only
recently have come up as a question, as natural
as it is. 
It should be emphasized that in \cite{Dirichletbook} many of the
important motivating questions (related to analysis
of Dirichlet series) require the use of $H^{\p}$ spaces
beyond the standard cases of $\p=1,2,\infty$.

In one variable, there is
not much to discuss while in several variables
we can have competing boundary zeros in the numerator
and denominator of a rational function.
It is natural to apply conformal maps in order
to study the problem more locally on a flat surface.  
If we map $(\R\cup\{\infty\})^2 \to \T^2$ and $(0,0)$ to $(1,1)$
via
\[
(x,y) \mapsto (z,w) = \left(\frac{1+ix}{1-ix}, \frac{1+iy}{1-iy}\right)
\]
then we can switch to studying a rational
function $Q/P$, where $P$ has no zeros in $\uhp^2 = \{(x,y): \Im (x), \Im (y) >0\}$, 
the bi-upper half plane, yet $P(0,0)=0$.
Now we ask if $Q/P$ is locally in $L^{\p}$
in the sense of being $L^{\p}$-integrable restricted to a neighborhood of $(0,0)$ 
in $\R^2$; 
 when this happens we write $Q/P \in L^{\p}_{loc}$.  
The approach taken here is to fix $P$ and then characterize
all $Q$ such that $Q/P \in L^{\p}_{loc}$.
In this paper, we give a complete solution to this problem in the
case of an isolated zero with respect to $\R^2$.
We are able to obtain a necessary and sufficient test for
membership in $L^{\p}_{loc}$ (Theorem \ref{finalint})
as well as a complete description of all $Q$
with $Q/P \in L^{\p}_{loc}$ (Theorem \ref{uberthm}).
The case of a boundary curve of zeros is interesting
but we do not pursue it here (see Remark \ref{rem:curve}).
We employ many of the results from \cite{BKPS} on a
local theory of stable polynomials; however, some 
novel extensions of this theory are also required
and developed here (Section \ref{sec:locface}). 
We do not go beyond two variables in this paper.  
The main obstructions to a theory in more variables is the
inapplicability of Puiseux series and B\'ezout's theorem 
as well as a more technical fact 
that certain stable polynomials (the \emph{real} stable 
polynomials) in two variables have analytic branches through a 
singularity (see the
related result Theorem \ref{facerealbranches}) while
this result does not carry over to three or more variables.
See \cite{BPSAJM} for a study of various things
that can happen in higher variables.

The objects of study here bear some similarities with 
log-canonical thresholds and multiplier ideals in algebraic 
geometry but the nature and focus of our results do not seem to connect 
directly with this theory (see \cite{LCTsurvey}).  
Motivated by the study
of oscillatory integrals there are also
a variety of results on studying integrability of inverses
of real analytic functions \cites{PS, Collins}, real analytic ratios \cite{Pramanik},
 or inverses of smooth functions \cite{Greenblatt}
 using Newton diagrams associated to the functions.
While the situation we study is more special
in some ways, we are also able to obtain an essentially
complete understanding; not just characterizing 
integrability but also obtaining precise descriptions of functions with
prescribed integrability.  The approach we follow is not 
short but it does not involve any hard analysis 
of breaking up a neighborhood of $(0,0)$ in $\R^2$ 
into regions where our function has 
certain behavior.  Our softer approach leads to 
simple algebraic conditions for integrability.

\subsection{First Example}

It is helpful to begin with the simplest non-trivial
example.

\begin{example} \label{favex}
Consider $P(z,w) = 2-z-w$ which has no zeros in $\D^2$ but a zero at $(1,1)$.
Given $Q \in \C[z,w]$, we ask when is $\frac{Q(z,w)}{P(z,w)}$ in $L^\p(\T^2)$?  The answer in this case
turns out to be relatively simple. For $\p < 3/2$ this is automatic, 
for $3/2\leq \p < 3$, we merely need $Q(1,1) = 0$, 
while for $\p\geq 3$ we need the additional first order condition $\frac{\partial Q}{\partial z}(1,1) = \frac{\partial Q}{\partial w}(1,1)$.
To get some sense why, we switch to the upper half plane setting via $z = \frac{1+ix}{1-ix}$, $w = \frac{1+iy}{1-iy}$ and clear
denominators to get
\[
(1-ix)(1-iy)P(z(x),w(y)) = -2i(x+y - i2xy)
\]
and it is enough to consider the local $L^{\p}$ integrability of rational functions with
the denominator $x+y-2ixy$.   
This polynomial can be written in terms of a parametrization of its zero set
\[
x+y-2ixy = (1-2ix)\left(y + \frac{x}{1-2ix}\right) =(1-2ix)(y+ x + i2x^2 + \text{ higher order})
\]
The factor $1-2ix$ does not vanish at $0$, so it is unimportant.  
The main factor is bounded above and below by $y+x+2ix^2$.
Thus, when analyzing integrability we can study
\[
\int_{(-\epsilon, \epsilon)^2} \frac{|Q(x,y)|^{\p}}{|y+x+2ix^2|^{\p}} dx dy \approx
\int_{(-\epsilon, \epsilon)^2} \frac{|Q(x,y-x)|^{\p}}{|y+2ix^2|^{\p}} dx dy
\]
by employing the transformation $y \mapsto y-x$.  Here $\approx$ just means both sides
are finite or both are infinite.  (Note also that the change of variables to the upper half
plane
introduces factors that are non-vanishing at $0$, so they can be disregarded 
for this problem.)
To proceed
we write $Q(x,y-x) = h(x,y) = h_0(x) + y h_1(x,y)$.  
Since $\frac{y}{y+2ix^2}$ is bounded on $\R^2\setminus \{(0,0)\}$, we need only consider $h_0(x)$
and its order of vanishing.
In effect we are only considering the
order of vanishing of $Q(x,-x)$.

The rest can be accomplished with the technical fact that
\[
\frac{x^{j}}{y+2ix^2} \in L^{\p}_{loc}
\]
if and only if $j > 2 - \frac{3}{\p}$.  For $\p < 3/2$, we have $2-3/\p<0$,
so the condition is $j \geq 0$ in this range (i.e.\ vacuous).
For $3/2\leq \p < 3$, we have $0\leq 2-3/\p < 1$, so the condition is $j \geq 1$ (i.e. $Q(0,0)=0$).
For $\p \geq 3$, the condition is $j \geq 2$, which amounts to 
\[
0=\left.\frac{d}{dx} Q(x,-x)\right|_{x=0} = \frac{\partial Q}{\partial x}(0,0) - \frac{\partial Q}{\partial y}(0,0) = 0. \qquad\qquad \diamond
\]
  \end{example}
This example outlines some aspects of our approach, however
the analysis becomes much more interesting when $P$
has multiple branches with
various orders of contact between 
distinct branches.

\subsection{Main Theorems}

An important first step in our approach is that we can
replace our polynomial $P$ with a simpler polynomial.
In the example above
we replaced $x+y-2ixy$ with $y+x+2ix^2$
and something similar can be done in general.
First, we define
\[
\refl{P}(x,y) = \overline{P(\bar{x},\bar{y})}
\]
which makes sense for any two variable polynomial.

In the following theorem we assume $P$ and $\refl{P}$
have no common factors---this
puts us into the situation of
finitely many (hence isolated) zeros in $\R^2$.
The case of boundary zeros filling out a curve in $\R^2$
is again interesting but not the focus
here. See Remark \ref{rem:curve} below.

The following was proven in \cite{BKPS}.

\begin{restatable}{theorem}{brackthm}[Theorem 1.2 \cite{BKPS}] \label{brackthm}
Let $P(x,y) \in \C[x,y]$ have no zeros in $\uhp^2$ and no 
common factors with $\refl{P}$.
Suppose $P$ vanishes to order $M$ at $(0,0)$.
Then there exist natural numbers $L_1,\dots, L_M \geq 1$
and real coefficient polynomials 
$q_1(x),\dots, q_M(x) \in \R[x]$
satisfying
\begin{itemize}
\item $q_j(0)=0$
\item $q_j'(0) > 0$
\item $\deg q_j(x) < 2L_j$ 
\end{itemize}
for $j=1,\dots, M$
such that if we define
\begin{equation} \label{pbrack}
[P](x,y) = \prod_{j=1}^{M} (y+ q_j(x) + i x^{2L_j})
\end{equation}
then
\[
P/[P] \text{ and } [P]/P
\]
are bounded in a punctured neighborhood of 
$(0,0)$ in $\R^2$.
\end{restatable}

Because of this theorem, we can replace $P$
which might be irreducible or even locally irreducible
with $[P]$, a product of simple factors
that capture the relevant geometry of the zero
set of $P$ near $(0,0)$.

\begin{restatable}{definition}{IP} \label{IP} 
For $\p \in [1,\infty]$, we define the ideal
\[
\mathcal{I}^\p_{P} = \{Q \in \C\{x,y\}: Q/P \in L^\p_{loc}(0)\}.
\]
Here $\C\{x,y\}$ is the set of convergent power series in $(x,y)$ 
centered at $(0,0)$ and $L^{\p}_{loc}(0)$ refers to the locally $L^{\p}$
functions; namely, those $f$ defined on a neighborhood $U$ of $(0,0)$ in $\R^2$ which  
belong to $L^{\p}(U)$ with respect to Lebesgue measure.  
We will often write $L^{\p}_{loc} = L^{\p}_{loc}(0)$ since
the ``$0$'' will be implied throughout.
\end{restatable}

Note that $\mathcal{I}^{\p}_{P} = \mathcal{I}^{\p}_{[P]}$ by Theorem \ref{brackthm}.

Our first main theorem is a complete characterization of the elements
of $\mcI_{P}^{\p}$.  To state it we need Definitions \ref{Onq} and \ref{Oij}.

\begin{restatable}{definition}{Onq}  \label{Onq}
For any $Q(x,y) \in \C\{x,y\}$, $q(x)\in \C[x]$, and $n \in \mathbb{N}$, let 
\[
O(n,q,Q)
\]
denote the highest power $k$ such that $x^k$ divides $Q(x, x^ny - q(x))$.
\end{restatable}

\begin{example} \label{Onqex}
For example, if $Q = (y + x + x^2 + i x^4)(y+x +i x^4)$
then
\[
O(3, x, Q(x,y)) = 5
\]
since $Q(x, x^3 y - x) = (x^3y + x^2 + i x^4)(x^3y + i x^4)$ 
is divisible by $x^5$.
\end{example}

\begin{restatable}{definition}{Oij} \label{Oij}
Given data $L_1,\dots, L_M \in \mathbb{N}$, $q_1(x),\dots, q_M(x) \in \R[x]$
we define
\[
O_{ij} = \min\{ \text{Ord}(q_j(x)-q_i(x)), 2L_j, 2L_i\}.
\]
Here ``Ord'' denotes the order of vanishing of a 
one variable polynomial (or analytic function) at the origin. 
In particular, $O_{jj} = 2L_j$.
\end{restatable}

Here is our characterization of integrability.

\begin{restatable}{theorem}{finalint} \label{finalint}
Let 
\begin{equation} \label{charP}
P(x,y) = \prod_{j=1}^{M} (y+q_j(x)+ i x^{2L_j})
\end{equation}
where $L_1,\dots, L_M \in \mathbb{N}$ and $q_1,\dots, q_M \in \R[x]$ with 
$\deg q_j < 2L_j$, $q_j(0)=0$. Let $1\leq \p < \infty$ and $Q \in \C\{x,y\}$.  
Then, $Q \in \mcI_{P}^{\p}$
if and only if for $j=1,\dots, M$,
\[
O(2L_j, q_j, Q) \geq \sum_{i=1}^{M} O_{ij} - \left\lceil \frac{2L_j+1}{\p} \right\rceil +1
\]
\end{restatable}

The main point is that integrability of $Q/P$ can
be characterized by measuring the generic
order of vanishing of $Q$ along
certain types of curves $y + q(x)+ x^{2L} s = 0$  
where $s$ is a ``generic'' parameter.

\begin{example} \label{favex2}
Returning to Example \ref{favex}, which we
reduced to the case of $P = y+x + ix^2$,
we have $M=1$, $L_1=1$, $q_1(x) = x$.
We can check when $Q(x,y)/P(x,y) \in L^{\p}_{loc}$ 
by computing $O(2, x, Q(x,y))$.
This equals the power of $x$ dividing 
$Q(x, x^2y - x)$.
It also helps to compute 
that 
\[
\left\lceil \frac{3}{\p} \right\rceil  = \begin{cases} 3 & \text{ for } 1\leq \p < 3/2\\
2 & \text{ for } 3/2 \leq \p < 3 \\
1& \text{ for } \p \geq 3\end{cases}
\]
In this case, $\sum_{i,j=1}^{1} O_{ij} = 2L_1 = 2$ so $Q/P \in L^{\p}_{loc}$
if and only if
\[
O(2, x, Q) \geq 3 - \left\lceil \frac{3}{\p} \right\rceil  = \begin{cases} 0 & \text{ for } 1\leq \p < 3/2\\
1 & \text{ for } 3/2 \leq \p < 3 \\
2 & \text{ for } \p \geq 3\end{cases}.
\]
The case $O(2,-x,Q) \geq 0$ is vacuous.
The case $O(2,-x,Q) \geq 1$ simply means $Q(0,0)=0$.
The case $O(2,-x,Q) \geq 2$ means $x^2$ divides $Q(x,x^2y-x)$.

This means $Q(0,0)= 0$ and $x$ divides
\[
\frac{d}{dx} Q(x,x^2y-x) = Q_x(x,x^2y-x) + Q_y(x,x^2y-x)(2xy-1)
\]
or rather $Q_x(0,0) - Q_y(0,0) = 0$.  This matches what was said
in Example \ref{favex} $\diamond$
\end{example}

As an application of Theorem \ref{finalint} we can prove that 
derivatives of locally bounded rational functions belong to $L^1_{loc}$.  
(See Theorem \ref{Derthm} in Section \ref{Dersec}.)
Translating this back to the bidisk/two-torus setting we obtain the following.

\begin{theorem} \label{elloneder}
If $f(z,w)$ is rational as well as bounded and analytic on $\D^2$, then the partial derivatives
\[
\frac{\partial f}{\partial z}, \frac{\partial f}{\partial w}
\]
belong to $H^1(\D^2)$.
\end{theorem}

Such a result is trivial in one variable, but if 
we do not assume $f$ is rational then nothing like this is true even in
one variable.  The singular inner function $f(z) = \exp\left(-\frac{1+z}{1-z}\right)$
(which in particular is bounded)
has $f'(z) = \frac{-2}{(1-z)^2} f(z) \notin H^1(\D)$. 
The conclusion of Theorem \ref{elloneder} cannot
be caused by any sort of bounded embedding since monomials 
$z^n$ have sup norm $1$ but the $L^1$ norm of the derivative equal to $n$. 
The theorem should be of interest because
although $H^1(\D^2)$ is widely studied
in connections with Hankel operators and harmonic
analysis, it is non-trivial to construct
elements of $H^1(\D^2)$ in a natural manner---most constructions
are ad hoc (i.e. products of $H^2$ functions).
The result is related to work of Bickel-Pascoe-Sola \cites{BPS18, BPS20}.
In those papers it is proven 
that given a rational \emph{inner} function $f:\D^2 \to \D$,
then 
$\frac{\partial f}{\partial z} \in H^{\p}(\D^2)$ if and only 
if $\p < 1+1/K$ where $K$ is a geometric quantity associated to
the denominator of $f$ (called the contact order).  
We will have a similar result along these lines but it
will not be an if and only if because a general bounded
rational function could behave better than this (see Theorem \ref{Derthm}).
In any case, by these results 
$H^1$ in Theorem \ref{elloneder} cannot be improved
to any $H^{\p}$ for $\p>1$.

To discuss our next main result 
consider the ideal $([P],\refl{[P]})$ generated by $[P]$ and $[\refl{P}]$
in $\C\{x,y\}$.  Elements of this ideal evidently belong
to $\mathcal{I}^{\p}_{P} = \mathcal{I}^{\p}_{[P]}$ so it makes sense to study the quotient 
$\mathcal{I}^{\p}_{P}/([P],[\refl{P}])$.  
This is finite dimensional because the larger set
\begin{equation} \label{IntQuot}
\C\{x,y\}/([P],[\refl{P}])
\end{equation}
is finite dimensional.  
In fact, the dimension of \eqref{IntQuot} is equal to the \emph{finite} 
 intersection multiplicity of the common zero at $(0,0)$ of
 $[P]$ and $[\refl{P}]$; see Fulton \cite{Fulton} Chapter 3, especially Section 3.3.
 Referring to Definition \ref{Oij}, it turns out that
 \begin{equation} \label{dims}
 \dim \C\{x,y\}/([P],[\refl{P}]) = \dim \C\{x,y\}/(P,\refl{P}) = \sum_{i,j} O_{ij} = \sum_{j} 2L_j + 2 \sum_{i<j} O_{ij}.
 \end{equation}

One of our main results is the following.

\begin{theorem} \label{l2dim}
Assume $P \in \C[x,y]$ has no zeros in $\uhp^2$, $P(0,0)=0$, and $P, \refl{P}$ have
no common factors.  Then,
\[
\dim( \mathcal{I}^2_P/([P],[\refl{P}])) = \frac{1}{2} \dim(\C\{x,y\}/(P,\refl{P}))
\]
\end{theorem}

Necessarily, the intersection multiplicity is even. 
Thus, if we disregard combinations of $[P],[\refl{P}]$, then ``half'' of the remaining power series are
square integrable.  A non-local version of this result was obtained in \cite{gKintreg}.
Along the way we obtain an
explicit basis for $\mathcal{I}^2_P/([P],[\refl{P}])$.
Combined with our test for membership in $\mathcal{I}^2_P$, this
 gives us a nearly complete understanding of $\mathcal{I}^2_{P}$.

Computing the dimension of the ideals corresponding to $L^{\p}$ is more subtle.
Define for $\p \in (0,\infty]$
\[
\mathcal{I}^{\p}_{P} = \{Q \in \C\{x,y\}: Q/P \in L^{\p}_{loc}(0)\}.
\]
Here $L^{\p}_{loc}(0)$ is the $\p$ analogue of $L^2_{loc}(0)$.
Again, $([P],[\refl{P}]) \subset \mcI_{P}^{\p}$ 
so it makes sense to consider the quotient $\mcI_{P}^{\p}/([P],[\refl{P}])$ 
and compute its dimension.  

\begin{restatable}{theorem}{lpdim} 
\label{lpdim}
Assume $P \in \C[x,y]$ has no zeros in $\uhp^2$, $P(0,0)=0$, and $P, \refl{P}$ have
no common factors.  Let $[P](x,y)$ and  
$L_1,\dots, L_M \in \mathbb{N}$ be defined
as in Theorem \ref{brackthm} and $O_{ij}$ defined 
as in Definition \ref{Oij}.
For $\p \in [1,\infty)$
\[
\dim \left(\mcI_{P}^{\p}/([P],[\refl{P}]) \right)= 
\sum_{(j,k):j<k} O_{jk} 
+ \sum_{k=1}^{M} \left(\left\lceil \frac{2L_k + 1}{\p}\right\rceil -1\right).
\]
and
\[
\dim\left( \mcI_P^{\infty}/([P],[\refl{P}]) \right)
= \sum_{(j,k):j<k} O_{jk}.
\]
\end{restatable}

Using \eqref{dims} it is possible
to express the above dimensions entirely
in terms of $\p$, $L_1,\dots, L_M$, and
$\dim \C\{x,y\}/(P,\refl{P})$; however,
the above expressions might be
the best way to see how the dimensions
change with respect to $\p$.
For examples, when $\p=1$ the dimension is 
\[
\sum_{(j,k):j<k} O_{jk}  + \sum_{k=1}^{M} 2L_k
\]
and for $\p=2$ it is
\[
\sum_{(j,k):j<k} O_{jk}  + \sum_{k=1}^{M} L_k
\]
which is $1/2$ of \eqref{dims}. 

\begin{example} 
Returning to Examples \ref{favex},\ref{favex2}, 
we have $[P] = y+ x + ix^2$, $[\bar{P}] = y + x - ix^2$
so that $([P],[\bar{P}]) = (y+x, x^2)$ after
manipulating combinations of $[P],[\bar{P}]$. 
Every $Q \in \C\{x,y\}$ can be reduced mod $([P],[\bar{P}])$
to a degree one polynomial $Q(x,y) = a+bx$.
According to Example \ref{favex2}, 
we therefore have $\mcI_P^{\p}/([P],[\refl{P}])$
is two dimensional for $1\leq \p < 3/2$,
one dimensional for $3/2\leq \p < 3$,
and zero dimensional for $\p \geq 3$.
This matches the computation
\[
\dim\left( \mcI_P^{\p}/([P],[\refl{P}]) \right)
= \left\lceil \frac{3}{\p} \right\rceil - 1. \quad \diamond
\]
\end{example}

Because dimensions are integers,
we always have the property that if $Q/P \in L^{\p}_{loc}$, then  $Q/P \in L^{\p+\epsilon}_{loc}$ for 
some $\epsilon>0$.
In particular, setting $K = \max\{2L_1,\dots, 2L_M\}$,
if $Q/P \in L^{1}_{loc}$ then 
$Q/P \in L^{\p}_{loc}$ for $\p < 1+1/K$.
On the other hand,
if $Q/P$ belongs to $L^{K+1}_{loc}$ then $Q/P$ belongs to $L^{\infty}_{loc}$.  
In the course of proving Theorem \ref{lpdim} 
we are able to give an explicit basis for $\mcI_P^{\p}/([P],[\refl{P}])$. 
We are also able to give a new
proof of the theorem of Koll\'ar \cite{kollar} which completed
a characterization of the ideal $\mathcal{I}_{P}^{\infty}$ conjectured
in \cite{BKPS}.  

\begin{restatable}{theorem}{Iinf} [See Theorem 1.2 of \cite{BKPS}, Theorem 3 of \cite{kollar}] \label{Iinf}
Assume the setup of Theorem \ref{lpdim}.
The ideal $\mathcal{I}^{\infty}_{P}$ is equal to the
product ideal
\[
\prod_{j=1}^{M} (y+q_j(x), x^{2L_j}).
\]
\end{restatable}

Koll\'ar's proof used a theorem of Zariski about
integrally closed ideals in 2-dimensional regular rings, while
the proof given here uses relatively soft analysis and elementary 
theory of algebraic curves (Puiseux series and intersection multiplicity 
 at the level of Fulton \cite{Fulton}). 

\begin{remark} \label{rem:curve}
We make a final quick comment about the situation where
$P$ and $\bar{P}$ may have a common factor---again
this is the case where $P$ will have a curve of zeros in $\R^2$.
In this case, if $Q(x,y)$ is analytic and has no
factors (locally) in common with $P$, then $Q/P$ is not locally
integrable.  
Here is a rough argument.  
We can move to a smooth point of the zero set of $P$ on $\R^2$, 
where $P(x,y)$ would look locally like $y+\phi(x)$
with $\phi(x)$ analytic with real coefficients.
When studying the integral $\int_{(-\epsilon,\epsilon)^2} |Q/P|$
we can apply a shear transformation 
and replace $y+\phi(x)$ (and hence $P$) with $y$.  
Any terms involving $y$ in $Q(x,y)$ can be disregarded since
they will be integrable, and we are left with analyzing
\[
\int_{(-\epsilon,\epsilon)^2} \left| \frac{Q(x,0)}{y}\right| dxdy
\]
which is infinite unless $Q(x,0) \equiv 0$.  
This would imply $y$ divides $Q$.  

Obtaining a full understanding of integrability in this situation
would force us to use $L^{\p}$ spaces for $\p<1$
and there would be some difficulties associated
with what happens when there are multiple factors
with curve and isolated zeros through $(0,0)$ in $\R^2$.  
$\diamond$
\end{remark}

\begin{remark}
Since this article first appeared several papers appeared
with connections to the work here.
The papers Bickel-Hong \cite{BH}, Tully-Doyle et al \cite{Tully} explore
the use of adjacency matrices for graphs
as a way to construct stable polynomials
with certain boundary behavior.  
The paper \cite{BKPS2} studied bounded rational functions
on the polydisk in more than 2 variables with smoothness
assumptions on the denominator stable polynomial.
\end{remark}

\section*{Acknowledgments}
A big thanks goes to my collaborators Kelly Bickel, James Pascoe,
and Alan Sola for listening to early versions of some of these
arguments.  A bonus thanks goes to Alan Sola for commenting 
on a draft of the paper.  Also, I offer my sincere gratitude
to the anonymous referee for a thoughtful and detailed report
that improved the paper in numerous ways.

 \section{Outline of our approach}
Section \ref{secfreq} contains restatements of the major 
theorems and definitions of the paper as well as some 
frequently used intermediate results so that readers need not 
always track them down in the body of the paper.  One could attempt to
simply read Section \ref{secfreq} to get a rapid summary of the technical
content of the paper.  The current section is a more
leisurely outline of the paper.

As we did above, the first step is to replace $P$ with the simpler $[P]$ as in \eqref{pbrack}.
Since $P/[P]$ is bounded above and below near $(0,0)$ in $\uhp^2$
(and also in $\R^2$), this replacement does not change the set of $Q(x,y)$ 
such that $Q/P$ belongs to $L^{\p}_{loc}$.  

The next step is to tackle the $L^2_{loc}$ case via a one variable ``quadrature'' or interpolation formula.
A simple version of the formula we refer to is the following; 
we need a more general version later.

\begin{theorem} \label{simplePW}
Given monic $p(y) \in \C[y]$ with no zeros in
$\overline{\uhp}$, let $A = \frac{1}{2}(p+\refl{p})$, $B = \frac{1}{2i}(p - \refl{p})$.
Then, $A$ has $M=\deg p$ distinct real zeros $a_1,\dots, a_M$
and for any $Q \in \C[y]$ with $\deg Q < M$ we have
\[
\int_{\R} \left|\frac{Q}{p}\right|^2 \frac{dy}{\pi} = \sum_{j=1}^{M} \frac{|Q(a_j)|^2}{B(a_j)A'(a_j)}
\]  
\end{theorem}

See Section \ref{onevarsec}.  
This is really a sampling/interpolation
 formula for the simplest type of 
de Branges space of entire functions, but the formula can
be deduced from scratch for readers unfamiliar with this background.

A parametrized version of this formula allows us to
express integrability in two variables in terms
of integrability of several one variable terms.
The trade-off is that we obtain useful duality techniques and algebraic
formulas in exchange for having to do a ``global'' integral with 
respect to $y$.  This is only a real problem for studying $L^1$
but it is surmountable.

Starting with $Q(x,y) \in \C\{x,y\}$ we can apply the Weierstrass
division theorem to write $Q(x,y) = Q_0(x,y) + Q_1(x,y) P(x,y)$
where 
\[
\deg_y Q_0(x,y) < \deg_y P(x,y).
\] 
We can then simply replace $Q$ with $Q_0$ and assume
$Q$ has this degree constraint.
We then employ a parametrized version of Theorem \ref{simplePW} to our 
two variable polynomial $P(x,y)$ to express integrals
\[
\int_{\R} \left|\frac{Q(x,y)}{P(x,y)}\right|^2 \frac{dy}{\pi}
= \sum_{j=1}^{M} \frac{|Q(x,a_j(x))|^2}{B(x,a_j(x))A_y(x,a_j(x))}
\]
in terms of evaluation at the roots $a_1(x),\dots, a_M(x)$ of $y\mapsto A(x,y)$.
Here $A_y = \frac{\partial A}{\partial y}$.
Then,
\[
\int_{(-\epsilon,\epsilon)\times \R} |Q/P|^2 \frac{dx dy}{\pi} 
= \sum_{j=1}^{M} \int_{(-\epsilon,\epsilon)} \frac{|Q(x,a_j(x))|^2}{B(x,a_j(x))A_y(x,a_j(x))} dx
\]
The roots $a_j(x)$ turn out to be analytic as a function of $x$ and integrability of
the terms on the right then only depends on the order of vanishing of 
\[
\frac{|Q(x,a_j(x))|^2}{B(x,a_j(x))A_y(x,a_j(x))}.
\]
This allows us to understand integrability on $(-\epsilon,\epsilon)\times \R$.
The contribution
from $\{(x,y): |x|<\epsilon, |y|\geq \epsilon\}$ is unimportant because for small
$x$ and large $y$, $Q/P$ behaves at worst like $1/y$ which is square integrable away from
$0$.
Putting this together we get that $Q/P$ is in $L^2_{loc}$ if and only if for every $j$
\begin{equation} \label{oov}
2\text{Ord}(Q(x,a_j(x))) \geq \text{Ord}(B(x,a_j(x))) + \text{Ord}(A_y(x,a_j(x))).
\end{equation}
where $\text{Ord}$ refers to the order of vanishing at zero of a one variable analytic function.
This gives a concrete condition for checking membership
in $L^2_{loc}$ in specific examples,
and with some modifications it can be used as a part of
our general approach to describing 
$\mcI^{2}_{P}$. 

Our next goal is to describe a basis
for $\C\{x,y\}/(P,\refl{P})$.  
By digging into the explanation for a formula
for the intersection multiplicity of two curves from
Fulton \cite{Fulton} (section 3.3), 
we can give an explicit basis. 
For $1\leq k \leq M$, let 
\[
F_k(x,y) = \frac{A(x,y)}{\prod_{j=1}^{k} (y-a_j(x))} = \prod_{j=k+1}^{M}(y-a_j(x))
\]
and $m_k = \text{Ord}B(x,a_k(x))$.
Then,
\begin{equation} \label{simplebasis}
\{ x^j F_k(x,y): k=1,\dots, M,\ 0\leq j <m_k\}
\end{equation}
is a basis for $\C\{x,y\}/(P,\refl{P})$.  More precisely, the
representatives of these power series in the quotient form a basis
for the quotient.
In particular, 
$\dim \C\{x,y\}/(P,\refl{P}) 
 = \sum_{k=1}^{M} m_k$.
Thus, each element $Q \in \C\{x,y\}/(P,\refl{P})$ can be
represented as
\[
Q(x,y) = \sum_{k=1}^{M} c_k(x) F_k(x,y) 
\]
where $c_k(x) \in \C[x]$, $\deg c_k < m_k$.  A problem
now is to see exactly what \eqref{oov} 
says about the order of vanishing of  the polynomials $c_k(x)$.

To analyze this we need a
companion to Theorem \ref{simplePW}, namely the representation formula
\[
Q(y) = \sum_{j=1}^{M} \frac{Q(a_j)}{A'(a_j)} \frac{A(y)}{y-a_j}
\]
which in parametrized form becomes
\begin{equation} \label{paramform}
Q(x,y) = \sum_{j=1}^{M} \frac{Q(x,a_j(x))}{A_y(x,a_j(x))} \frac{A(x,y)}{y-a_j(x)}.
\end{equation}
If we represent the basis elements $F_k(x,y)$ via this formula
then we can express the coefficients $Q(x,a_j(x))/A_y(x,a_j(x))$
in terms of the basis coefficients $c_k(x)$.  
At this point we use \eqref{oov} to exactly determine how integrability
affects the coefficients $c_k$. This makes it possible to determine
the dimension of $\mathcal{I}^2_{P}/(P,\refl{P})$.

\begin{example} \label{easy2branch}
Let us consider the example $P(x,y) = (x+y - ixy)(2x+y - ixy)$.  
It locally factors into
\[
P = (1-ix)^2(y+ x + i x^2 + \cdots)(  y + 2x + i2x^2+\cdots)
\]
and therefore $[P]$ is $(y+x+ix^2)(y+2x+ix^2)$.
From here we replace $P$ with $[P]$ and proceed to compute
\[
A = y^2 + 3xy +2x^2 - x^4,\ B = x^2(2y+3x).
\]
Locally
\[
A = (y-a_1(x))(y-a_2(x)) = (y + x - x^2 + \cdots)(y+2x + x^2 + \cdots).
\]
Note that $\text{Ord}(B(x,a_j(x))) = 3, \text{Ord}(A_y(x,a_j(x))) = 1$ for $j=1,2$.
Thus, $Q/P \in L^2_{loc}$ if and only if
\[
\text{Ord}(Q(x,a_j(x))) \geq 2
\]
for $j=1,2$. 

Next, we construct a basis for $\C\{x,y\}/(P,\refl{P})$ and
see what these constraints impose.
In the construction above, $F_1(x,y) = y-a_2(x) = y + 2x+x^2+\cdots$,
$F_2(x,y)=1$, $m_1 = \text{Ord}(B(x,a_1(x))) = 3$, $m_2 = 3$.
Therefore, our basis for $\C\{x,y\}/(P,\refl{P})$ is
\[
x^j (y-a_2(x)) \text{ for } 0\leq j < 3, \text{ and } x^j \text{ for } 0\leq j < 3.
\]
Thus, the dimension of $\C\{x,y\}/(P,\refl{P})$ is $6$ and 
we can represent an element as
\[
Q(x,y) = c_1(x)(y-a_2(x)) + c_2(x) 
\]
where $c_1(x),c_2(x)$ are polynomials with degree less than $3$.
In this simple example, we can directly check our
condition for $Q/P\in L^2_{loc}$, namely
that $Q(x,a_2(x)) = c_2(x)$ vanishes to order at least $2$
and $Q(x,a_1(x)) = c_1(x)(a_1(x)-a_2(x)) + c_2(x)$
also vanishes to order at least $2$.
Since $a_1(x)-a_2(x)$ vanishes to order $1$ we must
have $c_1(x)$ vanishes to order at least $1$.
Thus we have whittled the dimension from $6$ down to $3$,
the dimension of $\mathcal{I}^2_{P}/(P,\refl{P})$. $\diamond$
\end{example}

This outline works well for some 
examples; however for a general argument
 there is a subtle problem.
In order to obtain a precise dimension count easily expressible in
terms of the data contained in $[P]$, we need to be able to
compute the $m_j = \text{Ord}(B(x,a_j(x)))$ 
as well as $\text{Ord}(a_j(x)-a_k(x))$ for
$j\ne k$ and relate these to $[P]$.  
Unfortunately, we cannot \emph{always} do this.
But,
if we replace $A$ with $A+tB$ for a generic $t$, we can.
There is a way to generically match up the
branches of $A+tB$ with those of $P$.  
For the polynomial in Example \ref{easy2branch},
we do not even need to specify ``generically'' because
locally
\[
A+tB = (y + x + O(x^2))(y+ 2x + O(x^2))
\]
for every $t\in \R$.  However, in general
there can be exceptional values of $t$
as the following example shows.

\begin{example} \label{exex}
Consider $P = (y+x+ix^2)(y+x+ix^4)$.  Then,
\[
A+tB = (y+x)^2 + t x^2(1+x^2)(y+x) -x^{6}
\]
The value $t=0$ is exceptional because
\[
A = (y+x+x^3)(y+x-x^3)
\]
and there is no correspondence between the branches of $P$ and $A$;
both branches of $A$ fit the pattern of $y+x + O(x^2)$ but neither fit the pattern $y+ x+ O(x^4)$.
For $t\ne 0$ we can solve for the branches of $A+tB$
\[
A+tB = (y +x+  t x^2 + (t+t^{-1})x^4 + O(x^5))(y+x - (1/t)x^4 + O(x^5))
\]
and here we \emph{do} have a clear correspondence to the branches of $P$.  $\diamond$
\end{example}

Something like this generic correspondence 
was already established 
in \cite{BKPS} for polynomials with no zeros
in $\uhp^2$ (see Theorem 2.20 of \cite{BKPS}).  However, since we have replaced
$P$ with $[P]$ we cannot directly apply 
theorems from \cite{BKPS}.  Fortunately,
$[P]$ satisfies the weaker but still global
property of having no zeros in $\R\times \uhp$.

The formal part of the paper now begins with a local
theory of polynomials with no zeros on $\R\times \uhp$.
We then discuss a more general version of Theorem \ref{simplePW}
involving $A+tB$ and its associated two variable version.
Next, we construct a basis for $\C\{x,y\}/(P,\refl{P})$ built out of $A+tB$.
Then, we use the representation formula
associated to Theorem \ref{simplePW} to determine the
dimension of $\mathcal{I}^2_{P}/(P,\refl{P})$.
Along the way we take a few detours to determine
$L^{\p}$ conditions on the coefficients in \eqref{paramform}
and its $A+tB$ analogue.

\section{Local theory for polynomials with no zeros on $\R\times \uhp$} \label{sec:locface}

Proceeding as in our outline, our goal is to 
understand $P$ with no zeros in $\uhp^2$, however
our local model $[P]$ from Theorem \ref{brackthm}
merely has no zeros in $\R\times \uhp$.
Therefore, we study polynomials or even analytic functions
with this weaker property. 
There is some good secondary motivation for studying
this weaker property.  As shown in the papers Geronimo-Iliev-Knese \cite{GIK}, Knese \cite{semi},
non-vanishing on $\R\times \uhp$ is related to an associated
homogeneous polynomial being \emph{hyperbolic} with respect to a single
direction.  Hyperbolic polynomials are a generalization of stable polynomials; see 
for example \cite{hyperbolic} for general
information.  

Our goal in this section is to prove the following.

\begin{restatable}{theorem}{facebranches} \label{facebranches}
    Suppose $P\in \C\{x,y\}$ has no zeros in $\R\times \uhp$ and no factors in common
    with $\refl{P}$.  
Let $M = \text{Ord}(P(0,y))$. Then, $P$ satisfies
    \begin{equation} \label{Pfactored}
    P(x,y) = 
    u(x,y) \prod_{j=1}^{M} (y + q_j(x) + x^{2L_j} \psi_{j}(x^{1/k}))
    \end{equation}

    where $u \in \C\{x,y\}$ is a unit, $k\geq 1$ and 
    for $j=1,\dots M$
    \begin{itemize}
        \item $L_j \in \mathbb{N}$ 
        \item $q_j \in \R[x], q_j(0)=0$, $\deg q_j < 2L_j$,
        \item $\psi_j(x) \in \C\{x\}$, $\Im \psi_j(0) >0$. 
        \end{itemize}
 \end{restatable}

This is basically Theorem 2.16 of \cite{BKPS} 
for polynomials with no zeros in $\uhp^2$ with
the hypotheses weakened to
``no zeros in $\R\times \uhp$'' and the only weakening of the conclusion
 is that we cannot say $q_j'(0)>0$, which holds in the case of $P$ having
no zeros in $\uhp^2$.  
Note that if $P$ has no
zeros in $\uhp^2$, then any zeros in $\R\times \uhp$
must occur as part of ``vertical'' lines $\{(x_0,y):y\in \C\}$ for $x_0 \in \R$. 
(This follows from Hurwitz's theorem---for $t>0$, $x_0\in\R$, 
the polynomial $y\mapsto P(x_0+it,y)$
has no zeros in $\uhp$ and we can send $t\searrow 0$ to see
either $P(x_0,y) \equiv 0$ or $y\mapsto P(x_0,y)$ has no zeros in $\uhp$.) 
So, if $P$ 
and $\refl{P}$ have no common factors, then $P$
has no zeros on $\R\times \uhp$.
On the other hand, if $P$ and $\refl{P}$ do have
common factors then we cannot rule out
factors of the form $(x-x_0)$.
Since we are only interested in local properties
we need only consider the possibility of
a factor of the form $x^N$.  In any case,
the hypothesis of no zeros in $\R\times \uhp$ that
we now are using conveniently avoids these possibilities.

We can also establish
\begin{theorem}\label{facerealbranches}
If $P(x,y) \in \C\{x,y\}$ 
has no zeros in the region $\R \times (\C\setminus \R)$
then the zero set of $P$ is locally a union of smooth curves parametrized by 
 analytic functions with real coefficients.
\end{theorem}

Let us proceed into the details of Theorems \ref{facebranches} and \ref{facerealbranches}.
First, in either theorem we can apply the Weierstrass preparation theorem (since 
$P(0,y) \not \equiv 0$) and write
\[
P(x,y) = u(x,y) P_0(x,y)
\]
where $P_0(x,y) = \C\{x\}[y]$ is a Weierstrass polynomial of degree
$M := \text{Ord}(P(0,y))$ in $y$ and $u(x,y) \in \C\{x,y\}$
is non-vanishing at $(0,0)$.
As in the case of $\uhp^2$ studied in \cite{BKPS}, 
we can analyze the zero
set of $P_0(x,y)$ near $(0,0)$ 
via the Newton-Puiseux theorem and produce a factorization 
of $P_0$ into Puiseux series.
The branches of $P$ going through $(0,0)$
can be locally parametrized via $t\mapsto (t^N, -\phi(t))$
for analytic $\phi$.
For the branches to avoid $\R\times \uhp$,
we must have the property that whenever $t^N \in \R$ then
$\phi(t) \in \overline{\uhp}$.  Therefore, the following
proposition is really about describing
the branches of polynomials
with no zeros in $\R\times \uhp$.

\begin{prop} \label{locruhp}
Let $N\geq 1$ and let $\phi(t) \in \C\{t\}$ with $\phi(0)=0$.
If $\phi(t) \in \overline{\uhp}$ whenever $t^N\in \R$ then
there exists $\psi_0(t) \in \R\{t\}$ and $L\geq 1$ such that
\[
\phi(t) = \psi_0(t^N) + i b t^{2LN} \psi_1(t)
\]
where $b\geq 0$, $\psi_1\in \C\{t\}$, $\psi_1(0)=1$.  

Moreover, $\phi$ has the property that $\phi(t) \in \R$ whenever $t^N \in \R$
if and only if $b=0$ above; namely, 
\[
\phi(t) = \psi_0(t^N)
\]
and the $\psi_1$ term is unnecessary.
\end{prop}

In the case where $b>0$, 
we can absorb terms of $\psi_0(t^N)$ of degree 
greater than $2LN$ into the term $\psi_1(t)$ in order to rewrite
\[
\phi(t) = q(t^N) + t^{2LN} \psi(t)
\]
where $q\in \R[t]$, $\deg q < 2L$, $q(0)=0$, and $\Im\, \psi(0) > 0$.  
Looking at Theorem \ref{facebranches} where $P$ has no factors in common 
with $\refl{P}$, we see that all of the branches must have this form with $b>0$, namely
\[
0=y + q(x) + x^{2L}\psi(x^{1/N})
\]
where $\Im \psi(0) >0$
because otherwise we would have a whole family of branches with
real coefficients and this would mean $P$ and $\refl{P}$ would have a
common factor.
Theorem \ref{facebranches} is then a consequence of 
Proposition \ref{locruhp} because the conclusion can be rewritten in the form of
a factorization.  

In the case of $b=0$ we can dispense with Puiseux series
and obtain a factorization into analytic curves.
If $t\mapsto (t^N, -\phi(t))$ is an irreducible Puiseux parametrization with the
property that ``whenever $t^N$ is real then $\phi(t)$ is real'', 
then by Proposition \ref{locruhp} and
the assumption of irreducibility we see that $N=1$ and $\phi$ has only real coefficients.
Such a branch is simply the graph of an analytic function with real coefficients.
If all branches have this property, then we can break up our variety into smooth
branches.  This is the situation of Theorem \ref{facerealbranches},
which is therefore also a consequence of Proposition \ref{locruhp}.

To prove Proposition \ref{locruhp} we need the following lemma.

\begin{lemma}\label{proglemma}
Let $d>0$ and $a \in [0,1]$.
Suppose we have an arithmetic progression that avoids every other open unit interval:
\[
d \ints + a \subset [0,1] + 2\ints = \bigcup_{j\in \ints} [2j, 2j+1].
\]
Then, either 
\begin{itemize}
\item $d$ is an even integer $2m$ or
\item $d$ is an odd integer $2m+1$ and $a \in \{0,1\}$.
\end{itemize}
Both of these situations produce arithmetic progressions avoiding
every other closed unit interval so that this is a complete characterization.
\end{lemma}

\begin{proof}
We must have $d\geq 1$.  Indeed, if $b$ is the smallest element of $d\ints + a$
 greater than $1$
then $b\geq 2$ while $b-d \leq 1$ so that $d \geq 1$.

There exists $n$ such that $a+d-2n \in [0,1]$ so that $|a- (a+d-2n)|  = |d-2n| \leq 1$.
But then 
\[
|d-2n|\ints + a \subset [0,1] + 2\ints
\]
and this again implies either $d=2n$ or $|d-2n| \geq 1$ (so $|d-2n|=1$).
In the first case $d$ is even as claimed and in the second case $d$ is odd.
When $d$ is odd, say $d=2m+1$, then $a+2m+1-2m \in [0,1] +2\ints$.
So either $a=0$ or $a=1$.  
\end{proof}

\begin{proof}[Proof of Proposition \ref{locruhp}]
Write $\phi(t) = \sum_{j = 1}^{\infty} a_j t^j$ and
define
\[
\psi_0(t^N) := \sum_{j=1}^{\infty} \Re(a_{jN}) t^{jN}.
\]
We can replace $\phi(t)$ with
\[
\phi(t) - \psi_0(t^N)
\]
since this does not affect the hypotheses of the
proposition.  
If $\phi$ is now identically zero we are finished
and can take $b=0$.
So, suppose $\phi$ is not zero and
 let $a_k t^k$ be the first nonzero term of $\phi(t) = a_k t^k  + \text{ higher order}$.
The condition ``$t^N\in \R$ implies $\phi(t) \in \overline{\uhp}$''
implies that $k$ is a multiple of $N$.
Indeed, for $t = r e^{i\pi m/N}$ where $m \in \ints$
and $a_k = |a_k| e^{i\pi \alpha}$ with $\alpha \in [0,2)$
\[
\lim_{r \searrow 0} \frac{\phi(r e^{i\pi m/N})}{r^k} = |a| e^{i\pi(\alpha + km/N)} \in \overline{\uhp}
\]
so that
\[
mk/N + \alpha \in [0,1] + 2 \ints
\]
for $m\in \ints$.
In particular, $\alpha \in [0,1]$.
In other terms, 
\[
(k/N)\ints + \alpha \subset [0,1] + 2 \ints,
\]
i.e.\ an arithmetic progression avoiding every other open
unit interval.
By Lemma  \ref{proglemma}, $k/N=n$ must be a positive integer.
Note that we have removed the real parts of the coefficients of $\phi$
for powers of $t$ that are multiples of $N$.  We conclude
$a_k$ is purely imaginary; i.e. $a_k=ib$ for $b\in \R$.
Thus, $\alpha = 1/2$ (so $b>0$) and again by Lemma \ref{proglemma} 
we must have $k/N$ is even, say $k = 2LN$ (since the case
where $k/N$ is odd is ruled out by the condition $a \in \{0,1\}$ in the lemma).
Altogether $\phi(t)$ has the form
\[
\phi(t) = ib t^{2LN} + \text{ higher order}
\]
with $b>0$ and this is exactly the form of $ib t^{2LN} \psi_1(t)$ with $\psi_1(0)=1$ as
claimed in the proposition.

Finally, for $t^N\in \R$, the imaginary part of $\phi(t)$ is $t^{2LN}\Im(ib \psi_1(t))$
and this is non-zero unless $b=0$.  Thus, $\phi(t)$ is real whenever $t^N$ is
real exactly when $\phi$ has the claimed form.
\end{proof}

One of the main applications of Theorem \ref{facebranches} 
is the computation of the intersection multiplicity of $(0,0)$ as a
common zero of $P$ and $\refl{P}$.
We shall use standard formulas and the approach from
Fulton \cite{Fulton} (Chapter 3) to do this now, however later on 
(Section \ref{sec:basis}) we need
an explicit basis for $\C\{x,y\}/(P,\refl{P})$ that allows us to address
integrability and therefore we essentially rederive some of
the foundational material in \cite{Fulton} for completeness or lack of a 
suitable reference.

Let $N(Q_1,Q_2)$ be the intersection multiplicity at $(0,0)$
of $Q_1,Q_2 \in \C\{x,y\}$.
By standard formulas in \cite{Fulton}, we have
\[
N(P,\refl{P}) = \dim(\C\{x,y\}/(P,\refl{P})) =  
\sum_{i,j=1}^{M}\text{Ord}(q_j(x)+ x^{2L_j}\psi_{j}(x^{1/k}) 
- q_i(x) - x^{2L_i}\overline{\psi_i(\bar{x}^{1/k})}).
\]
For a fixed pair $i,j$, 
since $\Im \psi_j(0), \Im \psi_i(0) > 0$, the summand equals
\[
\min\{ \text{Ord}(q_j(x)-q_i(x)), 2L_j, 2L_i\}.
\]
Let us recall the following notation.

\Oij*

Note now that
\[
N(P,\refl{P}) = \sum_{i,j=1}^{M} O_{ij} = \sum_{j=1}^{M} 2L_j + 2 \sum_{(i,j): i<j} O_{ij}
\]
which is evidently even.

\subsection{Matching Branches}

For this section we \emph{do} need to assume $P$ 
is a polynomial and not simply analytic. 
If $P$ is analytic we can form $[P]$
just as in Theorem \ref{brackthm} and
work with $[P]$ instead.

As discussed in the outline, the real and imaginary
parts of $P \in \C[x,y]$ with no zeros on $\R\times \uhp$
and no factors in common with $\refl{P}$
have a variety of useful properties.
Recall the
real and imaginary coefficient polynomials
\[
A = \frac{1}{2}(P+\refl{P}), \qquad B = \frac{1}{2i}(P-\refl{P}).
\]
If $M = \text{Ord}(P(0,y))$ it is
convenient to normalize the coefficient
of $y^M$ in $P(0,y)$ to be real so
that $\text{Ord}(A(0,y)) = M$ and
$\text{Ord}(B(0,y)) > M$.

It is not difficult to establish using 
Theorem \ref{onevarfacts} to come later 
and Hurwitz's theorem
for $t\in \R$ that $A+tB \in \R[x,y]$ has no zeros
in $\R\times (\C\setminus \R)$
with the exception of finitely many
vertical lines $x=x_0$ for $x_0\in \R$.
Since $\text{Ord}(A(0,y)) = \text{Ord}((A+tB)(0,y)) = M$,
$A+tB$ does not vanish
on the line $x=0$ and therefore
$A+tB$ is non-vanishing on a neighborhood
of $(0,0)$ intersected with $\R\times (\C\setminus \R)$.
Therefore, Theorem \ref{facerealbranches} applies
to $A+tB$ and we can factor
$A+tB$ into $M$ analytic branches $y+\phi(x)=0$ 
with $\phi(x) \in \R\{x\}$. 
What requires more work is showing
that generically with respect to $t \in \R$ 
these analytic branches match up 
in a precise way with the potentially
non-analytic branches of $P$.
Our goal for this section is the following.

\begin{restatable}{theorem}{matchbranches} \label{matchbranches}
Assume $P(x,y) \in \C[x,y]$ has no zeros in $\R\times \uhp$ and
no factors in common with $\refl{P}$.  Let $M = \deg(P(0,y))$
and we normalize $P$ so that the coefficient of $y^M$
is real.

    Consider the conclusion of Theorem \ref{facebranches}
    including the data $2L_1,\dots, 2L_M$, $q_1(x),\dots, q_M(x)$ associated to $P$.
    Let 
    \[
A = \frac{1}{2}(P+\refl{P}) \qquad B = \frac{1}{2i}(P-\refl{P}).
\]
    Then, for all but finitely many $t\in \R$, we can locally factor 
\begin{equation}\label{AtBfactor}
A(x,y) + t B(x,y) = u(x,y;t) \prod_{j=1}^{M} (y + q_j(x) + x^{2L_j} \psi_{j}(x;t))
\end{equation}
where for each fixed $t$, $u(x,y;t) \in \C\{x,y\}$ is a unit, $\psi_{j}(x;t) \in \R\{x\}$.
This decomposition cannot be improved in the sense that for any $i\ne j$ 
\begin{equation} \label{Oijeq}
\text{Ord}(q_j(x) - q_i(x) + x^{2L_j}\psi_{j}(x;t) - x^{2L_{i}} \psi_{i}(x;t)) 
= \min\{ \text{Ord}(q_j(x)-q_i(x)), 2L_j, 2L_i\}
\end{equation}
\end{restatable}

The right side equals the expression $O_{ij}$ from Definition \ref{Oij}.

The content of Theorem \ref{matchbranches} is subtle 
and we will carefully explain it below.
The quick explanation is that we have a reasonable
way to match up the branches of $A+tB$ with the branches
of $P$ in a way to avoid certain collisions.
This is expressed with the technical conclusion \eqref{Oijeq}.
In particular, if $i\ne j$ but $q_j=q_i$ and $2L_j=2L_i$ then 
the only way for \eqref{Oijeq} to hold is that
\[
\psi_{j}(0;t) \ne \psi_i(0;t).
\]
Something similar to this is proven in the paper \cite{Anderson} (see Lemma 5.8).  
We discuss this connection further in Remark \ref{Andersonremark}.

\begin{restatable}{definition}{proper} \label{proper}
We will refer to the $t\in \R$ for which \eqref{AtBfactor} and \eqref{Oijeq}
hold as \emph{proper}.
\end{restatable}

Recall from Example \ref{exex} that the exceptional ``non-proper'' behavior can 
indeed occur. 
The following corollary lets us precisely compute the order
of vanishing of $B$ along branches of $A+tB$.

\begin{restatable}{corollary}{Bvanish} \label{Bvanish}
  Assume the setup and conclusion of Theorem \ref{matchbranches}.
  Then, for proper $t \in \R$, $j=1,\dots, M$, and $O_{ij}$ as in Definition \ref{Oij},
  \[
    \text{Ord}(B(x, -q_j(x)- x^{2L_j} \psi_{j}(x;t))) = \sum_{i=1}^{M} O_{ij} = 2L_j + \sum_{i: i\ne j} O_{ij}
    \]
    \[
      \text{Ord}((A_y+tB_y)(x, -q_j(x) - x^{2L_j} \psi_{j}(x;t))) = \sum_{i: i\ne j} O_{ij}.
    \]    
\end{restatable}

\begin{proof}
  Since $P(x,-q_j(x) - x^{2L_j}\psi_{j}(x;t)) = (i-t) B(x,-q_j(x) - x^{2L_j}\psi_{j}(x;t))$
  we see that
  \[
    \text{Ord}(B(x, -q_j(x)- x^{2L_j} \psi_{j}(x;t))) =
      \sum_{i} \text{Ord}(q_i(x) - q_j(x) + x^{2L_i}\psi_i(x^{1/k}) - x^{2L_j}\psi_j(x;t)).
    \]
    Since $\Im\ \psi_i(0) > 0$ it is not hard to run through
    cases to see that the summand is equal to  $O_{ij}$ as in Definition \ref{Oij}.

    For the second claim, evaluating the partial derivative $A_y+tB_y$
    at a branch of $A+tB$ yields
    \[
      \text{Ord}((A_y+tB_y)(x, -q_j(x) - x^{2L_j} \psi_{j}(x;t)))
      =
      \sum_{i\ne j} \text{Ord} (q_i(x) - q_j(x) + x^{2L_i} \psi_i(x;t) - x^{2L_j}\psi_j(x;t))
    \]
    and the summands are again $O_{ij}$ by Theorem \ref{matchbranches}.
  \end{proof}

In \cite{BKPS}, the following method was used to track
the relationship between the branches of $P$ and the branches 
of $A+tB$.  
Let us call $q(x)\in \R[x]$ a (real) \emph{initial segment} (of a branch) of $P$
of order $n$ if $P$ has a branch $y + \phi(x) = 0$ such that
$q(x) - \phi(x)$ vanishes to order $n$ or higher.  
We can allow for $\phi(x)$ to be a Puiseux series,
say $\phi(x) = \phi_0(x^{1/N})$, as in \cite{BKPS}. 
Since the branches of $P$ have a clear form, namely $0=y+\phi(x)$ where
\[
\phi(x) =  q(x) + x^{2L}\psi(x^{1/N})
\]
where $q(x) \in \R[x]$, $\deg q(x) < 2L$, $\psi\in \C\{x\}$, $\Im \psi(0)>0$
(from Theorem \ref{facebranches}), we can observe the branch
of $P$ above has $q(x)$ as a \emph{maximal} real initial segment in the following sense:
for any other $\tilde{q}(x) \in \R[x]$,
\[
\text{Ord}(\phi(x) - \tilde{q}(x)) \leq 2L.
\]
This simply relies on the fact that $\Im \psi(0) \ne 0$.
There is an algebraic way to count the presence
of initial segments.
Let us recall

\Onq*

If $q(x) \in \R[x]$ then 
\[
O(n,q,P) = O(n,q, \refl{P})  \leq O(n,q, A+tB)
\]
with equality holding for all $t$ with at most one exception.

\begin{lemma}[Lemma 2.22 of \cite{BKPS}] \label{lem222}
Let $P(x,y) \in \C[x,y]$, $q \in \R[x]$, $n \in \mathbb{N}$. The quantity
\[
O(n+1,q,P)-O(n,q,P)
\]
counts the number of branches of $P$ (with multiplicity) such that $q$ is an initial
segment of order $n+1$.
For $A = (1/2)(P+\refl{P})$, $B =(1/(2i))(P-\refl{P})$,
this is equal to 
\[
O(n+1,q,A+tB) - O(n,q, A+tB)
\]
for all $t$ with at most $2$ exceptions and it counts the analogous quantity:
the number of branches of $A+tB$ with $q$ as an initial segment of 
order $n+1$.
\end{lemma}

This lemma is stated in \cite{BKPS} in the context of polynomials
with no zeros in $\uhp^2$ but this fact is never used in the
proof and the lemma carries over without trouble to
$P$ with no zeros in $\R\times \uhp$ (or
even to an arbitrary polynomial).
The (omitted) proof is accomplished by looking
at $P$ branch by branch.

\begin{example}
As an example, consider $Q = (y+x+x^2+ix^4)(y+x+ix^4)$.
We computed in Example \ref{Onqex} that $O(3,x,Q(x,y)) = 5$.
To compute $O(4,x,Q(x,y))$ we look at
\[
Q(x, x^4y-x) = (x^4y + x^2 + ix^3)(x^4y + i x^4)
\]
which is divisible by $x^6$, so $O(4,x,Q(x,y))=6$.
So, $O(4,x,Q(x,y))- O(3,x,Q(x,y)) = 1$ counts
the one branch with $x$ as an initial segment of 
order $4$ or higher.
Also, $O(2,x,Q) = 4$ so we see 
$O(3,x, Q)-O(2,x,Q) = 1$ and we count
the same branch with $x$ as an initial segment of
order $3$ or higher.
Finally, since $O(1,x,Q) = 2$ we have
$O(2,x,Q)-O(1,x,Q) = 2$ gives the count
of two branches with $x$ as an initial segment of
order $2$ or higher. $\quad \diamond$
\end{example}

Before we begin the proof of Theorem \ref{matchbranches}
we need some preliminary observations.
Consider the \emph{set} of real segment data attached to $P$
\[
D(P) := \{ (2L_j, q_j(x)):  j=1,\dots, M\} \subset  2\mathbb{N} \times \R[x];
\]
namely, the pairs occurring in the product \eqref{Pfactored}
in Theorem \ref{facebranches}.  
Given a pair $(2L,q) \in D(P)$,
define
\[
S(2L, q, P) = \{ c\in \R: q(x) + cx^{2L} \text{ is a segment of } P 
\text{ of order } 2L+1 \text{ (or higher)}\}
\]
which gives us a way to track when $q(x)$
is extendable to a different element of $D(P)$.

Let $\mathcal{B}_t$ be the set of local branches 
of $A+tB$, each written as $y+\phi(x)$ with $\phi(x) \in \R\{x\}$.
We want to keep track of the branches of $A+tB$ that begin
with $q(x)$ but do not begin with any longer segment 
from $D(P)$.
Consider for $(2L,q) \in D(P)$, the set $\mathcal{B}_t(2L,q)$
of branches $y+\phi(x)$ of $A+tB$ such that
\begin{itemize}
\item $\text{Ord}(\phi(x) - q(x)) \geq 2L$, and
\item for any $c\in S(2L,q,P)$, $\text{Ord}(\phi(x) - (q(x) + cx^{2L})) = 2L$.
\end{itemize}

\begin{prop} \label{branchprop}
Assume the setup of Theorem \ref{matchbranches}.
Then, for all but finitely many $t\in \R$,
\[
\mathcal{B}_t = \bigcup_{(2L,q) \in D(P)} \mathcal{B}_t(2L,q)
\]
is a partition of $\mathcal{B}_t$
and the cardinality $\# \mathcal{B}_t(2L,q)$ is the number of
occurrences of $(2L,q)$ in the list $(2L_1,q_1),\dots, (2L_M,q_M)$
from Theorem \ref{facebranches}.
\end{prop}

\begin{proof}
First, let us point out that the sets $\mathcal{B}_t(2L,q)$
are disjoint.
If $y+\phi(x) \in \mathcal{B}_t(2L,q)\cap \mathcal{B}_t(2\tilde{L},\tilde{q})$
for $(2L,q),(2\tilde{L},\tilde{q}) \in D(P)$
then without loss of generality $2L \leq 2\tilde{L}$
and since
\[
\text{Ord}(\phi(x) - q(x))\geq 2L,\ \text{Ord}(\phi(x) - \tilde{q}(x)) \geq 2L
\]
we must have $\text{Ord}(q(x) -\tilde{q}(x)) \geq 2L$.
If $2L=2\tilde{L}$ then $q=\tilde{q}$ and we are finished.
If $2L < 2\tilde{L}$, then the coefficient $c$ of $x^{2L}$ in $\tilde{q}$ 
belongs to $S(2L,q,P)$ and
\[
\text{Ord}(\phi(x) - (q(x) + cx^{2L}) \geq 2L+1
\]
contrary to $y+\phi(x) \in \mathcal{B}_t(2L,q)$.

Next, we explain why every branch belongs to one of these
sets.
Since every branch of $P$ has a maximal initial real segment,
$S(2L,q,P)$ is equal to 
\[
\{c \in \R: \exists (2\tilde{L},\tilde{q}(x)) \in D(P) \text{ such that } \text{Ord}(q(x) +cx^{2L} - \tilde{q}(x)) \geq 2L+1\}.
\]
Notice that if $(2L, q(x))$ is the maximal real initial segment datum of
the branch $0=y+\phi(x)$ of $P$ then for $c\in\R$
\[
\text{Ord}(\phi(x) - (q(x)+cx^{2L})) \leq 2L
\]
by maximality. 
The set $S(2L,q,P)$ gives us a way to algebraically count
the branches of $P$ with maximal real segment datum $(2L,q(x))$.
Indeed,
\[
M(2L,q,P):= O(2L, q, P) - O(2L-1, q, P) - 
\sum_{c \in S(2L,q, P)} O(2L+1, q+cx^{2L}, P) - O(2L, q+cx^{2L},P)
\]
counts all of the branches of $P$ with maximal initial real segment datum $(2L,q(x))$,
Since every branch of $P$ has maximal real segment associated to an element $D(P)$,
the number $M(2L,q,P)$ is exactly the number of occurrences of 
$(2L,q)$ in the list
\[
(2L_1, q_1),\dots, (2L_M, q_M)
\]
and therefore we can count all branches
\[
\sum_{(2L,q)\in D(P)} M(2L,q,P)  = M.
\]

By Lemma \ref{lem222}, the analogous quantity
\[
\begin{aligned}
M(2L, q, A+tB) := & O(2L, q, A+tB) - O(2L-1, q, A+tB)\\
& - 
\sum_{c \in S(2L, q, P)} O(2L+1, q+cx^{2L}, A+tB) - O(2L, q+cx^{2L},A+tB)
\end{aligned}
\]
counts the branches, say $y+\psi(x)=0$, of $A+tB$ such that 
$\text{Ord}(\psi(x) - q(x)) \geq 2L$ and 
for any $(2\tilde{L},\tilde{q}(x)) \in D(P)$ with $2\tilde{L}>2L$ and
$\text{Ord}(q(x) - \tilde{q}(x)) \geq 2L$ we have
\[
\text{Ord}(\psi(x) - \tilde{q}(x)) \leq 2L;
\]
namely, the branches beginning with $q(x)$ 
that do not extend by even one coefficient to
agree with a longer segment associated to $P$.
Thus,
\[
M(2L,q, A+tB) = \# \mathcal{B}_t(2L,q).
\]

We must have 
\[
M(2L, q , P) = M(2L,q, A+tB)
\]
for all but finitely many $t$ and as we vary over all 
$(2L,q) \in D(P)$.  
Therefore, for all but finitely many $t$
\begin{equation} \label{MAtB}
\sum_{(2L,q)\in D(P)} M(2L,q, A+tB) = M
\end{equation}
and this shows that $\mathcal{B}_t(2L,q)$ partitions the branches
of $A+tB$.
\end{proof}

We now
begin the proof of Theorem \ref{matchbranches}.
It is convenient to get ahead of ourselves a little bit
and use a couple of standard one variable results 
that we incidentally prove in the next section;
namely, Theorem \ref{onevarfacts}
which says that if $p(y) \in \C[y]$ has no zeros in $\overline{\uhp}$
then writing $p=A+iB$, for $t\in \R$ the polynomial 
$A+tB$ has only real simple zeros and 
$y\mapsto A(y)/B(y)$ is strictly increasing
on any interval of $\R$ not containing a zero
of $B(y)$.

\begin{proof}[Proof of Theorem \ref{matchbranches}]
Let us fix $t\in \R$ such that the conclusion
of Proposition \ref{branchprop} holds.
The factorization of $A+tB$ follows
from the partitioning of branches 
stated in Proposition \ref{branchprop}.
What is left to prove is the computation
of the order of contact between two branches of
$A+tB$.

It is worth pointing out
now that $A+tB$ cannot have
any repeated branches (although
it can certainly have repeated 
branch types from $D(P)$).
In fact, for fixed $x>0$ small enough, $y\mapsto P(x,y)$ has no zeros in $\overline{\uhp}$.
(Zeros on $\R^2$ are common zeros of $P$ and $\refl{P}$ and 
there are only finitely many of these.)
Then,  $y\mapsto A(x,y)+tB(x,y)$ 
has only real simple zeros
by the forthcoming Theorem \ref{onevarfacts}.  This 
precludes any repeated branches.

Now, we need to prove that for two branches of
$A+tB$ when written according to the partition in Proposition \ref{branchprop}---namely,
\[
0 = y + \underset{\phi(x)}{\underbrace{q(x) + x^{2L} \psi(x)}} \text{ and } 0= y+\underset{\tilde{\phi}(x)}{\underbrace{\tilde{q}(x) + x^{2\tilde{L}} \tilde{\psi}(x)}}
\]
where $(2L,q), (2\tilde{L},\tilde{q}) \in D(P)$,
with the first branch belonging to $\mathcal{B}_t(2L,q)$ and the second
belonging to $\mathcal{B}_t(2\tilde{L},\tilde{q})$---
that we have
\begin{equation}\label{Ords}
\text{Ord}(q(x) + x^{2L} \psi(x) - (\tilde{q}(x) + x^{2\tilde{L}} \tilde{\psi}(x))) = 
\min\{\text{Ord}(q(x) - \tilde{q}(x)), 2L, 2\tilde{L}\}.
\end{equation}

First, suppose $2L < 2\tilde{L}$.
If $\text{Ord}(q(x) - \tilde{q}(x)) < 2L$
then 
\[
\text{Ord}(q(x) - \tilde{q}(x) + x^{2L}\psi(x) - x^{2\tilde{L}} \tilde{\psi}(x)) = \text{Ord}(q(x)-\tilde{q}(x))
\]
and this agrees with the minimum on the right side of \eqref{Ords}.

If $\text{Ord}(q(x) - \tilde{q}(x)) \geq 2L$,
then $\tilde{q}(x) = q(x) + c x^{2L} + O(x^{2L+1})$
where $c\in S(2L,q,P)$.
Since $y+\phi(x) \in \mathcal{B}_t(2L,q)$ 
we cannot 
have $\psi(0)  = c$.
In this case we have that the left side of \eqref{Ords} is $2L$
and so is the right side.

In the case where $2L = 2\tilde{L}$ and $q(x) \ne \tilde{q}(x)$
we must have  
$\text{Ord}(q(x) - \tilde{q}(x)) < 2L$ which works
just like the first case above.

Finally, suppose $(2L,q) = (2\tilde{L},\tilde{q})$.
For $x\ne 0$ small enough the branches do not intersect and so for $x>0$ we have 
without loss of generality
\[
q(x) + x^{2L} \psi(x) > \tilde{q}(x) + x^{2\tilde{L}} \tilde{\psi}(x)
\]
(since otherwise the opposite inequality holds).
Notice that 
for each $x$, as $y$ increases from
$y_1 = -(q(x) + x^{2L} \psi(x))$ 
to
$y_2 = -(q(x) + x^{2L} \tilde{\psi}(x))$ 
the function $y\mapsto A(x,y)/B(x,y)$
runs from $-t$, increases to $\infty$, goes
from $-\infty$ to $-t$;
thus it attains every value except $-t$.  
This follows from Theorem \ref{onevarfacts}.
In particular, for $s \ne t$, $A+sB$ has a branch,
say  $y + \phi(x;s)$, in between $y_1$ and $y_2$.
For fixed $s$ we may have to decrease $x$ so that
this analytic branch exists.
Then,
\[
q(x) + x^{2L} \psi(x) > 
\phi(x;s) > q(x) + x^{2L} \tilde{\psi}(x)
\]
and the initial expansion of $\phi(x;s)$
must begin with $q(x)$.
We cannot have $\psi(0) = \tilde{\psi}(0)$
for then every branch in between would begin 
\[
q(x) + c x^{2L}
\]
for $c =  \psi(0)$.
This would then generically
be an initial segment of
$A+sB$ in which case it would be
an initial segment of $P$.
Indeed,
\[
  O(2L+2, q+cx^{2L},P) - O(2L+1, q+cx^{2L}, P)
\]
counts the number of branches of $P$ with $q+cx^{2L}$ as an
initial segment of order at least $2L+1$.
This generically equals the corresponding quantity for $A+sB$.
However, we deliberately chose $t$
so that its segments built from 
the data $(2L,q)$ do not extend to 
longer segments.  Therefore $\psi(0)\ne \tilde{\psi}(0)$ and then
the left and right sides of \eqref{Ords} equal $2L$.
This completes the proof.
\end{proof}

A useful addendum to the above proof is the following.
Consider two values 
of $t$, say $t=t_1, t_2$, for which Proposition \ref{branchprop} holds.
For a fixed datum $(2L,q) \in D(P)$,
we just proved above that 
$A+t_1 B$ and $A+t_2B$ have 
$M(2L,q,P)$ branches
that begin with $q(x)$ but that
do not extend by even one coefficient
to agree with any other branches of $P$.
Let us label the branches
\begin{equation} \label{labbranch}
q(x) + x^{2L}\psi_j(x;t_1),\ q(x) + x^{2L} \psi_j(x;t_2)
\end{equation}
for $j=1,\dots, M(2L,q,P)$,
where $M(2L,q,P)$, given in the proof of Proposition \ref{branchprop},
is the number of occurrences of $(2L,q)$ in the list $(2L_1,q_1),\dots, (2L_M,q_M)$.
We proved above that the values
$\{\psi_j(0;t_1): j=1,\dots, M(2L,q,P)\}$
are all distinct and likewise for $t_2$.
However we can also argue that
\[
\{\psi_j(0;t_1): j=1,\dots, M(2L,q,P)\} \cap 
\{\psi_j(0;t_2): j=1,\dots, M(2L,q,P)\} = \varnothing.
\]
Indeed if we had $c = \psi_j(0;t_1) = \psi_k(0;t_2)$
then there would be infinitely many
values of $s$ so that $A+sB$ has
a branch beginning $q(x) + c x^{2L}$.
This would imply that $P$
has a branch beginning $q(x)+cx^{2L}$
which again it does not
by choice of $t_1,t_2$.  

\begin{prop} \label{distval}
Assume the setup and conclusion of Theorem \ref{matchbranches}.
For any two distinct values of $t$, say $t=t_1,t_2$ 
such that Proposition \ref{branchprop} holds, consider the branches
of $A+t_1 B$ and $A+t_2 B$ with initial segment 
datum $(2L, q)$.  We write out these branches
as in \eqref{labbranch}.  
Then, the sets 
\[
S_1 = \{\psi_j(0;t_1): j=1,\dots, M(2L,q,P)\}, \quad
S_2 = \{\psi_j(0;t_2): j=1,\dots, M(2L,q,P)\}
\]
both contain $M(2L,q,P)$ real numbers 
and $S_1 \cap S_2 = \varnothing$. 
\end{prop}

This proposition will be important later when we
wish to show the order of vanishing 
of a polynomial $Q(x,y)$ along a branch of $A+tB$,
say $\text{Ord}(Q(x,a_j(x;t)))$, generically in $t$ 
only depends on the
branch type---i.e. which datum $(2L,q)$ the branch
is associated to.   

\begin{remark}\label{Andersonremark}
We mentioned earlier that Theorem \ref{matchbranches}
is similar to Lemma 5.8 of \cite{Anderson}.
This lemma states---in the language of this paper and of Theorem \ref{matchbranches}---that
there are at most finitely many 
$t$ such that for some $j$, $q_j(x) + x^{2L_j} \psi_{j}(0;t)$
is an initial segment of a branch of $P$ of order $2L_j+1$.

Recall that $\mathcal{B}_t(2L_j, q_j)$ is
the set of branches of $A+tB$ that begin with $q_j(x)$
but do not extend further along some datum $(2L,q) \in D(P)$.
Proposition \ref{branchprop} exactly says that
for all but finitely many $t$, the set of branches of $A+tB$
all have this property.  This implies Lemma 5.8 of \cite{Anderson}.

Thus,  Lemma 5.8 of \cite{Anderson} is very close
to Theorem \ref{matchbranches} and Proposition \ref{branchprop}
although it does not appear to directly give the further separation of
branches for distinct $t$ values given in Proposition \ref{distval} 
(which is not to suggest it needed or claimed to).  $\diamond$
\end{remark}

\section{One variable facts and formulas} \label{onevarsec}
The goal of this section is to review the following known 
facts and formulas in one variable.  We will apply them in the
next section to two variables where we will view the additional
variable as a parameter. 

\begin{theorem} \label{onevarfacts}
Assume $p(y) \in \C[y]$, $\deg p = M \geq 1$, and $p$ has no zeros in $\overline{\uhp}$.
Let $A = (p+\refl{p})/2$ and $B = (p-\refl{p})/(2i)$.  
Then, for each $t \in \R$, $A+tB$ has only simple real zeros,
and $B$ has only simple real zeros.
The zeros of $A+tB$ and $B$ interlace in the sense that between any
two zeros of $A+tB$ there is a zero of $B$ and vice versa.

For $y \in \R$, $B(y)A'(y) - A(y)B'(y)>0$ so that in particular
\[
\frac{A(y)}{B(y)}
\]
is strictly increasing  
on any interval in $\R$ not containing a zero of $B$.
In addition, $A(y)/B(y)$ has positive imaginary part for $y \in \uhp$.
\end{theorem}

The next theorem is our key integral formula and associated 
decomposition formula.  It is convenient to assume $p$ is monic.

\begin{theorem} \label{parseval}
Assume the setup and conclusion of Theorem \ref{onevarfacts}.
Also, assume that $p(y)$ is monic.
Let $a_1(t),\dots, a_M(t)$ denote the $M$
necessarily real roots of $A+tB$.
For any $Q \in \C[y]$ with $\deg Q< \deg p$
\[
\int_{\R} \left|\frac{Q}{p}\right|^2 \frac{dy}{\pi} = \sum_{j=1}^{M} \frac{|Q(a_j(t))|^2}{B(a_j(t))(A'(a_j(t)) + t B'(a_j(t)))} 
\]
and
\[
Q(y) =
\sum_{j=1}^{M} \frac{Q(a_j(t))}{A'(a_j(t)) + t B'(a_j(t))} \frac{A(y)+tB(y)}{y- a_j(t)}. 
\]
\end{theorem}

Note that from Theorem \ref{onevarfacts}, $B(a_j(t))(A'(a_j(t))+t B'(a_j(t))) > 0$
so that everything above makes sense.  Also, $\frac{A(y)+tB(y)}{y- a_j(t)}$ is a
polynomial since $a_j(t)$ is a zero of $A+tB$.

The proofs of Theorems \ref{onevarfacts} and \ref{parseval}
 use what are called polynomial de Branges spaces.
The theory is relatively simple so it 
makes sense to develop from scratch
rather than provide opaque references to literature
on more general de Branges spaces.  In any case, 
everything stated above can be found in \cite{debranges} and many
other references.

Let $p(y) \in \C[y]$ have no zeros in $\overline{\uhp}$ and $\deg p = M$.
Consider the $M$-dimensional Hilbert space $H = \{ Q \in \C[y]: \deg Q < M\}$ 
equipped with inner product
\[
(Q_1,Q_2) = \int_{\R} Q_1 \bar{Q}_2  \frac{dy}{\pi|p|^2}.
\]
Since $H$ is finite dimensional and polynomials are determined by
their values on $\R$, point evaluation at any $\eta \in \C$ is a bounded linear
functional on $H$.  Therefore, there exists a unique $K_{\eta} \in H$ such that 
\[
(Q, K_{\eta}) = Q(\eta).
\]
On the other hand, for any orthonormal basis $\{Q_j\}_{j=1}^{M}$ of $H$
one can directly check that 
\[
K_{\eta}(y) = \sum_{j=1}^{M} Q_j(y) \overline{Q_j(\eta)}.
\]
Thus, the function $K(y,\eta) = K_{\eta}(y)$, called the reproducing kernel,
is a polynomial in $y, \bar{\eta}$.  Also, note that $K(y,y)> 0$ for $M\geq 1$ since
we could choose a constant function to be in our orthonormal basis.  

We can establish an explicit formula for $K$.

\begin{prop}
Assume the setup of Theorem \ref{onevarfacts}.
Then,
\[
K(y,\eta) = \frac{p(y) \overline{p(\eta)} - \bar{p}(y) \overline{\bar{p}(\eta)}}{-2 i(y-\bar{\eta})}.
\]
\end{prop}

\begin{proof}
Set $L(y,\eta) =L_\eta(y)= \frac{p(y) \overline{p(\eta)} - \bar{p}(y) \overline{\bar{p}(\eta)}}{-2 i(y-\bar{\eta})}$
 which is evidently a polynomial in $y, \bar{\eta}$ of degree less than $p$ since
the numerator vanishes for $y=\bar{\eta}$.

Observe that 
\[
(Q, L_{\eta}) =  p(\eta) \underset{I_1}{\underbrace{\int_{\R} \frac{Q}{p} \frac{dy}{2\pi i (y-\eta)}}}  
- \bar{p}(\eta) \underset{I_2}{\underbrace{\int_{\R} \frac{Q}{\bar{p}} \frac{dy}{2\pi i(y-\eta)}}}.
\]
To evaluate $I_1$ we apply
the residue theorem to a large half-circle in $\uhp$ and send the radius to $\infty$.
Using the fact that $p$ has no zeros in $\overline{\uhp}$ one obtains $I_1= Q(\eta)/p(\eta)$.
To evaluate $I_2$ we apply the residue theorem to a large half-circle in the lower half plane
$-\uhp$ and now obtain $I_2=0$ because $\eta \notin -\uhp$.
Thus, $(Q,L_{\eta}) = Q(\eta)$, meaning $L_{\eta}$ has the unique
property that $K_{\eta}$ has.  Therefore, $K(y,\eta) = L(y,\eta)$ 
for $(y,\eta) \in \C\times \uhp$.  The formula continues to hold for
all $(y,\eta) \in \C^2$ since both expressions are polynomials.
\end{proof}

\begin{proof}[Proof of Theorem \ref{onevarfacts}]
First we rewrite $K$ in terms of $A+tB$ and $B$:
\begin{equation} \label{Ky}
K(y,\eta) = \frac{A(y) \overline{B(\eta)} - B(y)\overline{A(\eta)}}{y-\bar{\eta}} = 
\frac{(A(y)+tB(y)) \overline{B(\eta)} - B(y)(\overline{A(\eta)+t B(\eta)})}{y-\bar{\eta}}
\end{equation}
For $y \notin \R$ 
\begin{equation} \label{Kyynr}
K(y,y) 
= \frac{\Im((A(y)+t B(y)) \overline{B(y)})}{\Im(y)} > 0
\end{equation}
which implies $(A(y)+tB(y))/B(y)$ has positive imaginary part in $\uhp$
and negative imaginary part in $-\uhp$.  
In particular, $A+tB, B$ only have real roots.
For $y \in \R$,
\begin{equation} \label{Kyyr}
K(y,y) 
= (A'(y)+tB'(y))B(y) - B'(y)(A(y)+tB(y))> 0
\end{equation}
which implies $(A(y)+tB(y))/B(y)$ is locally strictly increasing in $y \in \R$ 
(away from zeros of $B$).
At a zero $a$ of $A+tB$ we have $(A'(a)+tB'(a))B(a) > 0$
so that all zeros of $A+tB$ are simple and do not coincide with
zeros of $B$.  Similarly, zeros of $B$ are simple.
Since $A(y)/B(y)$ is locally strictly increasing,
two consecutive zeros of $A+tB$ must have a zero
of $B$ between them.  This proves Theorem \ref{onevarfacts}.
\end{proof}
 
 \begin{proof}[Proof of Theorem \ref{parseval}]
 We pick up where the previous proof left off but now we also
 assume $p$ is monic.
Since $p$ is monic we have $\deg(A+tB) = \deg p = M$.
If the (necessarily real) $M$ roots of $A+tB$ are written $a_j(t)$ for $j=1,\dots, M$, in some order, then 
\[
(K_{a_j(t)})_{j=1}^{M}
\]
forms an orthogonal basis for $H$.
Orthogonality follows from 
\[
  (K_{a_j(t)}, K_{a_i(t)}) = K(a_i(t), a_j(t)) = 0
\]
via \eqref{Ky} for $i\ne j$.

Note that 
\[
K_{a_j(t)}(y) = B(a_j(t)) \frac{A(y)+tB(y)}{y- a_j(t)}
\]
and
\[
K(a_j(t), a_j(t)) = B(a_j(t))(A'(a_j(t)) + t B'(a_j(t))) > 0.
\]
In particular, this shows
\begin{equation} \label{Icalc}
\int_{\R} \left|\frac{A(y)+tB(y)}{y-a_j(t)} \right|^2 \frac{dy}{\pi|p(y)|^2} = \frac{A'(a_j(t))+t B'(a_j(t))}{B(a_j(t))}.
\end{equation}
Then, for each $t\in \R$
\[
k_j(y):= \frac{K_{a_j(t)}(y)}{\sqrt{K(a_j(t),a_j(t))}} = \sqrt{\frac{B(a_j(t))}{A'(a_j(t)) + t B'(a_j(t))} }\frac{A(y)+tB(y)}{y- a_j(t)}
\]
is an orthonormal basis for $H$ (running over $j=1,\dots M$). 
Theorem \ref{parseval} then follows via the Parseval identities
\[
(Q,Q) = \sum_{j=1}^{M} |(Q,k_j)|^2 \text{ and } Q(y) = \sum_{j=1}^{M} (Q, k_j) k_j.
\]
\end{proof}

\subsection{$L^{\p}$-theory in one variable}

We next attempt to get an $L^\p-\ell^\p$ version of the above results.
This seems to be relatively uncharted territory
for polynomial de Branges spaces
even though a variety of results are known
for one type of prototypical de Branges space, namely the Paley-Wiener 
spaces (see for instance \cite{LS}).

If $\deg Q < \deg p$, then $|Q(y)/p(y)| \leq \frac{const}{|y|+1}$  
so that $Q/p \in L^\p(\R)$ 
for $1<\p\leq\infty$.  
We would like to obtain an equivalence of norms
\[
\|Q/p\|_{L^{\p}} \approx (\sum_{j=1}^{M} |Q(a_j(t))|^{\p} w_j)^{1/\p}
\]
with some appropriate weights $w_j$.  To see what the weights should be
we test with $Q = \frac{A+tB}{y-a_j(t)}$ and get
\[
\int_{\R} \left|\frac{A(y)+tB(y)}{y-a_j(t)}\right|^\p \frac{dy}{\pi |p|^{\p}} \approx |A'(a_j(t)) + t B'(a_j(t))|^\p w_j. \]

\begin{restatable}{definition}{Ijtp} \label{Ijtp}
Set
\[
I_{j}(t,\p) := \left\| \frac{A+tB}{(\cdot-a_j(t)) p} \right\|_{L^{\p}(\R)}
\]
where we use $\frac{dy}{\pi}$ as the measure associated to $L^{\p}(\R)$. 
\end{restatable}

Note that \eqref{Icalc} shows
\begin{equation} \label{Icalc2}
I_j(t,2)^2 = \frac{A'(a_j(t))+t B'(a_j(t))}{B(a_j(t))}.
\end{equation}

From this heuristic we show
\[
\|Q/p\|_{L^{\p}} \approx \left(\sum_{j=1}^{M} \left| \frac{Q(a_j(t))}{A'(a_j(t)) + t B'(a_j(t))}\right|^{\p} I_{j}(t,\p)^{\p} \right)^{1/\p}.
\]
We can achieve this through some straightforward uses of Hölder's inequality and the $L^{\p}$
triangle inequality but the key is that the duality
of $L^2$ allows us to get lower bounds with 
crude control on the constants involved.

\begin{restatable}{prop}{onevarlp} \label{onevarlp}
  Let $p \in \C[y]$ be monic, have no zeros in $\overline{\uhp}$, and
   $\deg p = M$.
  Write $p = A+iB$, and let $a_1(t),\dots, a_M(t)$ denote the real roots of $A+tB$.  
  Define $I_j(t,\p)$ as in Definition \ref{Ijtp}.
  We use the measure $dy/\pi$ as the measure associated to $L^{\p}(\R)$.      

Then, for $1<\p<\infty$  and $Q \in \C[y]$ with $\deg Q < \deg p$ we have
\[
\|Q/p\|_{L^{\p}(\R)}
\leq M^{1/\p'} \left(\sum_{j=1}^{M} \left| \frac{Q(a_j(t))}{A'(a_j(t)) + t B'(a_j(t))}\right|^{\p} I_{j}(t,\p)^{\p} \right)^{1/\p}
\]
and
\[
\begin{aligned}
&\left(\sum_{j=1}^{M} \frac{|Q(a_j(t))|^{\p}}{|A'(a_j(t)) + t B'(a_j(t))|^{\p}}I_j(t,\p)^{\p}\right)^{1/\p} \\
&\leq 
\|Q/p\|_{L^{\p}(\R)}
\left(\sum_{j=1}^{M} \frac{|B(a_j(t))|^{\p}}{|A'(a_j(t))+t B'(a_j(t))|^{\p}} I_j(t,\p)^{\p} I_j(t,\p')^{\p}\right)^{1/\p}
\end{aligned}
\]
where $\frac{1}{\p} + \frac{1}{\p'} =1$.   \end{restatable}

\begin{proof}
There is no substantial difference in the proof for the case $t=0$
and the general case so we shall assume $t=0$.  
The roots $a_j(t)$ in the case $t=0$ will just be written as $a_j$.
Similarly, $I_j(0,\p)$ will just be written $I_j(\p)$.

By Theorem \ref{parseval} 
\[
\frac{Q(y)}{p(y)}  =  \sum_{j=1}^{M} \frac{Q(a_j)}{A'(a_j)} \frac{A(y)}{y-a_j} \frac{1}{p(y)}
\]
and therefore
\[
\| Q/p\|_{L^{\p}(\R)} \leq \sum_{j=1}^{M} |Q(a_j)/A'(a_j)| I_j(\p)
\leq \left(\sum_{j=1}^{M} |Q(a_j)/A'(a_j)|^{\p} I_j(\p)^{\p} \right)^{1/\p} M^{1/\p'}
\]
by $L^{\p}$-triangle inequality and H\"older's inequality for sums.

To prove our inequality in the opposite direction
we set $H(y) = \sum_{j=1}^{M} c_j K_{a_j}(y)$ where $K$
is the kernel function from before
\[
K_{a_j}(y) = B(a_j) \frac{ A(y)}{y-a_j}
\]
and the constants $c_j$ are to be determined.
We have
\begin{equation} \label{Qaj}
|\sum_{j=1}^{M} Q(a_j) \bar{c_j}| = |(Q, H)|   \leq \|Q/p\|_{L^{\p}(\R)} \|H/p\|_{L^{\p'}(\R)}.
\end{equation}
Now let 
\begin{equation} \label{cj}
c_j = \overline{\text{sgn}(Q(a_j))} \frac{|Q(a_j)|^{\p-1}}{|A'(a_j)|^{\p}} I_j(\p)^{\p}
\end{equation}
so that 
\begin{equation} \label{Qaj2}
|\sum_{j=1}^{M} Q(a_j) \bar{c_j}| = \sum_{j=1}^{M} |Q(a_j)/A'(a_j)|^{\p} I_j(\p)^{\p}.
\end{equation}
By the $L^{\p'}$ triangle inequality,
\begin{equation} \label{Hoverp}
\|H/p\|_{L^{\p'}(\R)} \leq \sum_{j=1}^{M} |c_j| \|K_{a_j}\|_{L^{\p'}(\R)}
\end{equation}
Now
\[
\|K_{a_j}\|_{L^{\p'}(\R)} = |B(a_j)| I_j(\p')
\]
and therefore by \eqref{cj} and \eqref{Hoverp}
\[
\begin{aligned}
\|H/p\|_{L^{\p'}(\R)} &\leq \sum_{j=1}^{M} \frac{|Q(a_j)|^{\p-1}}{|A'(a_j)|^{\p-1}}  I_j(\p)^{\p/\p'} \frac{|B(a_j)|}{|A'(a_j)|} I_j(\p) I_j(\p')\\
&\leq 
\left(\sum_{j=1}^{M} \frac{|Q(a_j)|^{\p}}{|A'(a_j)|^{\p}}  I_j(\p)^{\p} \right)^{1/\p'} \left(\sum_{j=1}^{M} \frac{|B(a_j)|^{\p}}{|A'(a_j)|^{\p}} I_j(\p)^{\p} I_j(\p')^{\p}\right)^{1/\p}.
\end{aligned}
\]
Putting this together with \eqref{Qaj} and \eqref{Qaj2} yields
\[
\left(\sum_{j=1}^{M} |Q(a_j)/A'(a_j)|^{\p} I_j(\p)^{\p} \right)^{1/\p} \leq \left(\sum_{j=1}^{M} \frac{|B(a_j)|^{\p}}{|A'(a_j)|^{\p}} I_j(\p)^{\p} I_j(\p')^{\p}\right)^{1/\p} \|Q/p\|_{L^{\p}(\R)}.
\]
\end{proof}

We state the case $\p=1$ separately.

\begin{prop} \label{onevarlone}
 Assume the setup of Proposition \ref{onevarlp}.
For $\delta>0$ and $j=1,\dots, M$, let 
\[
I_j^{\delta}(t) = \int_{|y|<\delta} \left| \frac{A(y)+tB(y)}{(y-a_j(t))p(y)}\right| \frac{dy}{\pi}. 
\]
Then, for $Q \in \C[y]$ with $\deg Q \leq M-1$ 
\[
\|Q/p\|_{L^1(-\delta,\delta)} \leq 
\sum_{j=1}^{M} \left|\frac{Q(a_j(t))}{A'(a_j(t)) + t B'(a_j(t))}\right| I_j^{\delta}(t)
\]
and if $\deg Q \leq M-2$ we have
\[
\sum_{j=1}^{M} \left|\frac{Q(a_j(t))}{A'(a_j(t)) + t B'(a_j(t))}\right| I_j^{\delta}(t)
\leq 
\|Q/p\|_{L^1(\R)} 
\sum_{j=1}^{M} \left|\frac{B(a_j(t))}{A'(a_j(t)) + t B'(a_j(t))}\right| I_j(t,\infty) I_j^{\delta}(t).
\]

\end{prop}

\begin{proof}
As in the previous proof the case $t=0$ is no different from 
arbitrary $t$.
The first inequality is just the $L^1$-triangle inequality as in the previous
proof.  For the second inequality we assume $\deg Q \leq M-2$
so that $Q/p$ belongs to $L^1(\R)$.  
We can follow the same proof as above except we set
\[
c_j = \overline{\text{sgn}(Q(a_j))} \frac{1}{|A'(a_j)|} I_j^{\delta}
\]
and $H = \sum_j c_j K_{a_j}$.  
Then,
\[
|(Q,H)|  = \sum_{j=1}^{M} \frac{|Q(a_j)|}{|A'(a_j)|} I_j^{\delta}
\leq \|Q/p\|_{L^1} \|H/p\|_{L^{\infty}}
\]
and now we simply use the $L^{\infty}$ triangle
inequality on $H$
\[
\|H/p\|_{L^{\infty}} \leq \sum_{j=1}^{M} \frac{|B(a_j)|}{|A'(a_j)|} I_j^{\delta} I_j(\infty).
\]
\end{proof}

\section{Characterization of local integrability}

With the one variable formulas in place, 
we use parametrized versions of them to
characterize local integrability.
Following our outline we work with a polynomial of 
the form \eqref{pbrack} instead of our original $P$
with no zeros on $\uhp^2$.

Our first goal is to prove the following characterization
of local square integrability.

\begin{restatable}{theorem}{ltwochar} \label{l2char}
Let 
\[ 
P(x,y) = \prod_{j=1}^{M} (y+q_j(x)+ i x^{2L_j})
\] 
where $L_1,\dots, L_M \in \mathbb{N}$ and $q_1,\dots, q_M \in \R[x]$ with 
$\deg q_j < 2L_j$ and $q_j(0)=0$.
Let $A = (1/2)(P+\refl{P}), B= (1/(2i))(P-\refl{P})$.
By Theorem \ref{facerealbranches}, $A+tB$
has $M$ smooth branches, labelled $y-a_j(x;t) = 0$ for $j=1,\dots, M$, passing through 
$(0,0)$.

Let $Q \in \C\{x,y\}$.  Then, $Q/P \in L^2_{loc}$ if and only if
for $j=1,\dots, M$
\begin{equation} \label{Qord}
2 \text{Ord}(Q(x,a_j(x;t))) \geq \text{Ord}(B(x,a_j(x;t))(A_y(x,a_j(x;t)) + t B_y(x,a_j(x;t)))).
\end{equation}
\end{restatable}

We can apply Theorem \ref{facerealbranches} above because
the polynomial $P$ has no zeros in $(\R\times \overline{\uhp})\setminus \{(0,0)\}$.
By Theorem \ref{matchbranches}, the order of vanishing on the right side 
can be expressed in terms of the data $2L_1,\dots, 2L_M, q_1(x),\dots, q_M(x)$
for all but finitely many $t$.  
Before doing this, we must order the branches of $A+tB$ so that 
they correspond to factors of $P$ as in Theorem \ref{matchbranches}.
Once this is done we have
\[
\text{Ord}(B(x,a_j(x;t))(A_y(x,a_j(x;t)) + t B_y(x,a_j(x;t)))) = 2L_j + 2 \sum_{i:i\ne j} O_{ij}
\]
with $O_{ij}$ defined as in Definition \ref{Oij}.
Our condition on $Q$ is then expressed as
\[
\text{Ord}(Q(x,a_j(x;t))) \geq L_j + \sum_{i:i\ne j} O_{ij}.
\]

\begin{proof}
Given $Q(x,y) \in \C\{x,y\}$,
by the Weierstrass division theorem we may write
\[
Q(x,y) = Q_0(x,y) + Q_1(x,y) (A+tB)(x,y)
\]
where $Q_0\in \C\{x\}[y]$ has degree less than $M$ in $y$.
Since $(A+tB)/P$ is bounded, for the purposes of studying integrability of $Q/P$ 
we may replace $Q$ with $Q_0$ and
assume the degree in $y$ is less than $M$.
Notice that this does not change the condition \eqref{Qord} either.

For $x\in \R\setminus \{0\}$, $y\mapsto P(x,y)$ has no zeros
in $\overline{\uhp}$
and we can apply the one variable theory from Section \ref{onevarsec}
to $y\mapsto P(x,y)$ with $x$ fixed. 
By Theorem \ref{parseval},
\begin{equation} \label{Qform}
Q(x,y) = \sum_{j=1}^{M} \frac{Q(x,a_j(x;t))}{A_y(x,a_j(x;t)) + t B_y(x,a_j(x;t))} \frac{A(x,y)+tB(x,y)}{y- a_j(x;t)}
\end{equation}
\[
\int_{\R} \left|\frac{Q(x,y)}{P(x,y)}\right|^2 \frac{dy}{\pi} = \sum_{j=1}^{M}
 \frac{|Q(x,a_j(x;t))|^2}{B(x,a_j(x;t))(A_y(x,a_j(x;t)) + t B_y(x,a_j(x;t)))}.
\]
Note that the denominator
\[
B(x,a_j(x;t))(A_y(x,a_j(x;t)) + t B_y(x,a_j(x;t)))
\]
is analytic in $x$ and positive for $x\ne 0$.
If we integrate over $x \in (-\epsilon, \epsilon)$
we see that 
\[
\int_{(-\epsilon, \epsilon)\times \R} |Q(x,y)/P(x,y)|^2 \frac{dx dy}{\pi} <\infty
\]
if and only if for $j=1,\dots M$
\[
\int_{(-\epsilon, \epsilon)} \frac{|Q(x,a_j(x;t))|^2}{B(x,a_j(x;t))(A_y(x,a_j(x;t)) + t B_y(x,a_j(x;t)))} dx <\infty
\]
and this holds if and only if
\[
2 \text{Ord}(Q(x,a_j(x;t))) \geq \text{Ord}(B(x,a_j(x;t))(A_y(x,a_j(x;t)) + t B_y(x,a_j(x;t)))).
\]
Finally, we must establish that
\begin{equation} \label{bigy}
\int_{|y|>\epsilon} \int_{|x|<\epsilon}  |Q(x,y)/P(x,y)|^2 dx dy < \infty
\end{equation}
without any conditions on $Q$ (beyond $\deg_y Q(x,y) < M$).
For $|x|<\epsilon$, $|Q(x,y)| \leq C |y|^{M-1}$
and $|P(x,y)| \geq c|y|^{M}$ for $y$ large enough.
Since $P(x,y)$ is non-vanishing for $(x,y) \ne (0,0)$, 
for intermediate ranges of $y$ we have that $P$ is bounded below.
Thus, \eqref{bigy} holds because for large $y$ the integrand is 
of the order $1/|y|^2$ and hence integrable at $\infty$.

Thus \eqref{Qord} is equivalent to local square integrability and 
proves the theorem. 
\end{proof}

\subsection{Characterization of local $L^{\p}$ integrability}

We can follow roughly the same strategy as we did for square integrability
except that Proposition \ref{onevarlp} 
which provides some type of equivalence
of $L^{\p}$ and $\ell^{\p}$ norms
in one variable has some dependence on $x$
when we put in $x$ as an additional parameter.
In order to control the constants as a
function of $x$ we must use a proper $t$ as
in Theorem \ref{matchbranches} thereby
allowing us to use Corollary \ref{Bvanish}.

Our goal is to prove the following.

\begin{restatable}{theorem}{lpchar} \label{lpchar}
Assume the same setup as Theorem \ref{l2char}.
Namely, we are given $P$ as in \eqref{charP} and associated
$A,B$ and branches $y-a_j(x;t)$ of $A+tB$.
We shall order the branches $y-a_j(x;t)$ so that they
correspond to the branches of $P$ as in Theorem \ref{matchbranches}; in particular,
we assume $t \in \R$ is proper (Definition \ref{proper}).

Let $1\leq \p < \infty$.
For $Q\in \C\{x,y\}$, $Q/P \in L^{\p}_{loc}$ if and only if
for $j=1,\dots, M$
\[
\text{Ord}(Q(x,a_j(x;t))) \geq \sum_{i=1}^{M} O_{ij} - \left\lceil \frac{2L_j+1}{\p} \right\rceil +1.
\]
\end{restatable}

The key to using Proposition \ref{onevarlp} is estimating the 
quantities now dependent on the parameter $x$:
\[
I_j(t,\p)(x) := \left\| \frac{A(x,\cdot)+tB(x,\cdot)}{(\cdot-a_j(x;t))P(x,\cdot)} \right\|_{L^{\p}(\R)} 
\]
and 
\[
\frac{|B(x,a_j(x;t))|^{\p}}{|A'(x,a_j(x;t)) + t B'(x,a_j(x;t))|^{\p}} I_j(t,\p)^{\p} I_j(t,\p')^{\p}.
\]

\begin{lemma} \label{Iorder}
Let $t\in \R$ be proper as in Theorem \ref{matchbranches}
and let $1<\p < \infty$.
Then, there exist constants $\epsilon, c, C>0$ such that
for each $j=1,\dots, M$ and $|x|<\epsilon$
\[
c x^{-2L_j(1-1/\p)} \leq I_j(t,\p)(x) \leq C x^{-2L_j(1-1/\p)}.
\]
For $\p =1$, $t\in \R$ proper, and $\delta>0$, there exist $\epsilon, c, C>0$ such that 
\[
c \leq \int_{(-\delta,\delta)} \left| \frac{A(x,y)+tB(x,y)}{(y-a_j(x;t))P(x,y)} \right| \frac{dy}{\pi} \leq C \log(1/|x|)
\]
for $|x|<\epsilon$.

For $\p = \infty$, $t\in \R$ proper, and $\delta>0$, there exist $\epsilon, c, C>0$ such that 
\[
c \log(1/|x|)^{-1} x^{-2L_j} \leq \sup_{|y|<\delta} \left| \frac{A(x,y)+tB(x,y)}{(y-a_j(x;t))P(x,y)} \right| \leq I_j(t,\infty) \leq C x^{-2L_j}
\]
for $|x|<\epsilon$.
\end{lemma}

\begin{proof}
For $t$ proper as in Theorem \ref{matchbranches} we note that
\[
\frac{A(x,y)+tB(x,y)}{(y-a_j(x;t))P(x,y)} 
= \frac{1}{y+q_j(x)+ix^{2L_j}} \prod_{k:k\ne j} \frac{y+q_k(x) +x^{2L_k} \psi_k(x;t)}{y+q_k(x) + ix^{2L_k}}.
\]
To bound the factors of the product we note
\[
\frac{y+q_k(x) +x^{2L_k} \psi_k(x;t)}{y+q_k(x) + ix^{2L_k}} -1 = 
\frac{x^{2L_k}(\psi_k(x;t) - i)}{y+q_k(x) + ix^{2L_k}}
\]
and have
\[
\left|\frac{y+q_k(x) +x^{2L_k} \psi_k(x;t)}{y+q_k(x) + ix^{2L_k}}\right| \leq
1 + \frac{x^{2L_k}(O(1))}{x^{2L_k}} \leq \text{ const}
\]
for $x$ small and $y\in \R$ arbitrary.
We conclude
\[
\left|\frac{A(x,y)+tB(x,y)}{(y-a_j(x;t))P(x,y)}\right| \leq \frac{C}{|y+q_j(x)+ix^{2L_j}|} 
\]
for $x$ small.  Let us label the frequently used expression
\[
R(x,y) := \frac{A(x,y)+tB(x,y)}{(y-a_j(x;t))P(x,y)};
\]
the dependence on $t$ and $j$ will only be important
insofar as $t$ is proper and fixed and $j$ comes 
from a finite set.

Now, for $\p >1$
\[
I_{j}(t,\p)(x)^{\p} \leq C \int_{\R} \frac{1}{|y+q_j(x) + i x^{2L_j}|^{\p}} dy
= C \int_{\R} \frac{1}{|y + i x^{2L_j}|^{\p}} dy 
\leq C x^{2L_j(1-\p)}
\]
for $x$ small. So, $I_{j}(t,\p)(x) \leq C x^{-2L_j(1-1/\p)}= C x^{-2L_j/\p'}$.
For $\p = 1$ we have
\[
\int_{(-\delta,\delta)} |R(x,y)| \frac{dy}{\pi}
\leq C \int_{(-\delta, \delta)} \frac{1}{|y+q_j(x)+ix^{2L_j}|} \frac{dy}{\pi}.
\]
By integrating over $|y+q_j(x)| < x^{2L_j}$ and $|y+q_j(x)| \geq x^{2L_j}$
we can bound the integral by $C \log(1/|x|)$.  
For the case $\p = \infty$, we have
\[
|R(x,y)| \leq C x^{-2L_j}
\]
since $|A/P|$, $|B/P|$ are bounded on $\R^2$.

To prove bounds below, we first use
 \eqref{Icalc2} and H\"older's inequality, for $\p >1$
\[
\left|\frac{A_y(x,a_j(x;t))+tB_y(x,a_j(x;t))}{B(x,a_j(x;t))}\right| = I_j(t,2)(x)^2  \
\leq I_j(t,\p)(x) I_j(t,\p')(x)  \leq C I_j(t,\p) x^{-2L_j/\p}
\]
using $I_j(t,\p') \leq C x^{-2L_j/\p}$.  
The left side behaves like $x^{-2L_j}$
by Corollary \ref{Bvanish}
so 
we see that
\[
I_j(t,\p)(x) \geq c x^{-2L_j(1-1/\p)}
\]
Thus, $I_j(t,\p)(x) \approx x^{-2L_j(1-1/\p)} = x ^{-2L_j/\p'}$.

To handle $\p=1,\infty$, we first point out that
\[
I_j(t,2)^2(x) \geq \int_{(-\delta, \delta)} |R(x,y)|^2 \frac{dy}{\pi}  - C \int_{|y|>\delta} \frac{1}{|y+q_j(x)|^2} dy
\]
and 
\[
\int_{|y|>\delta} \frac{1}{|y+q_j(x)|^2} dy \leq C \frac{1}{(\delta - |q_j(x)|)^2}
\]
which is simply bounded above for $x$ small since $q(0)=0$.
Therefore, 
\[
\begin{aligned}
I_j(t,2)^2(x) + O(1) &\leq \int_{(-\delta, \delta)} |R(x,y)|^2 \frac{dy}{\pi} \\
&\leq \int_{(-\delta,\delta)} |R(x,y)| \frac{dy}{\pi} \sup_{|y|<\delta} |R(x,y)|.
\end{aligned}
\]
If we use our $\p=\infty$ bound above, then we get the $\p=1$ bound below:
\[
I_j(t,2)^2(x) + O(1) \leq C x^{-2L_j} \int_{(-\delta,\delta)} |R(x,y)| \frac{dy}{\pi} 
\]
Since $I_j(t,2)^2(x) \approx x^{-2L_j}$ we see that
\[
\int_{(-\delta,\delta)}|R(x,y)| \frac{dy}{\pi}
\]
is bounded below by a positive constant for $x$ small enough.

If we use our $\p=1$ bound above, then we get the $\p=\infty$ bound below:
\[
I_j(t,2)^2(x) + O(1) \leq C \log(1/|x|) \sup_{|y|<\delta} |R(x,y)| 
\]
and in this case we get
\[
c x^{-2L_j} (\log 1/|x|)^{-1} \leq \sup_{|y|<\delta} |R(x,y)|.
\]
 \end{proof}

\begin{prop} \label{lpnorm}
Assume the setup of Theorem \ref{lpchar}.
Set for $j=1,\dots, M$
\[
r_j = 2L_j - \p \sum_{i=1}^{M} O_{ij}.
\]
For $t\in \R$ proper and sufficiently small $x$ we have that
for $Q \in \C\{x\}[y]$ with $\deg Q_y < M$ and $\p>1$
\[
\| Q(x,\cdot)/P(x,\cdot) \|_{L^{\p}(\R)}
\approx
\left(\sum_{j=1}^{M} |Q(x,a_j(x;t))|^{\p} |x|^{r_j}\right)^{1/\p}.
\]

For $t\in \R$ proper, $\delta >0$, and 
sufficiently small $x$ we have that
for $Q \in \C\{x\}[y]$ with $\deg Q_y < M-1$ and $\p=1$
\[
\|Q(x,\cdot)/P(x,\cdot)\|_{L^1(-\delta,\delta)} \leq C \log(1/|x|) \sum_{j=1}^{M} |Q(x,a_j(x;t))| |x|^{r_j}
\]
and
\[
c (\log(1/|x|)^{-1} \sum_{j=1}^{M} |Q(x,a_j(x;t))| |x|^{r_j} \leq \|Q(x,\cdot)/P(x,\cdot)\|_{L^1(\R)}
\]
\end{prop}

\begin{proof}
We simply insert the estimates
\[
|A_y(x,a_j(x;t)) + t B_y(x,a_j(x;t))| \approx |x|^{\sum_{i:i\ne j} O_{ij}}
\]
\[
|B(x,a_j(x;t))| \approx |x|^{\sum_{i=1}^{M}O_{ij}}
\]
\[
I_j(t,\p) \approx |x|^{-2L_j(1-1/\p)}
\]
from Corollary \ref{Bvanish} and Lemma \ref{Iorder}
into Proposition \ref{onevarlp}.
Indeed,
\begin{equation} \label{complicated}
\frac{|B(x,a_j(x;t))|^{\p}}{|A_y(x,a_j(x;t)) + t B_y(x,a_j(x;t))|^{\p}} I_j(t,\p)^{\p} I_j(t,\p')^{\p} \approx |x|^{r}
\end{equation}
for 
\[
r = \p(\sum_{i=1}^{M} O_{ij} - \sum_{i: i\ne j} O_{ij} -2L_j(1-1/\p) -2L_j/\p) = 0
\]
since $O_{jj} = 2L_j$.
In other words, for proper $t$ and $x$ small, \eqref{complicated}
is bounded above and below by constants.
The other $Q$ independent quantities in Proposition \ref{onevarlp}
can be calculated similarly.

For $\p=1$, using Proposition \ref{onevarlone},
Corollary \ref{Bvanish}, and Lemma \ref{Iorder} 
\[
\|Q(x,\cdot)/P(x,\cdot)\|_{L^1(-\delta,\delta)} 
\leq C \log(1/|x|) \sum_{j=1}^{M} |Q(x,a_j(x;t))| |x|^{-\sum_{i:i\ne j} O_{ij}}
\]
and similarly
\[
c \sum_{j=1}^{M} |Q(x,a_j(x;t))| |x|^{-\sum_{i:i\ne j} O_{ij}}
\leq \|Q(x,\cdot)/P(x,\cdot)\|_{L^{1}(\R)} C \log(1/|x|)
\]
since
\[
\left|\frac{B(x,a_j(x;t))}{A_y(x,a_j(x;t)) + t B_y(x,a_j(x;t))}\right|
\approx |x|^{2L_j}
\]
and $I_j(t,\infty) \leq C |x|^{-2L_j}$.   Note that
$r_j = -\sum_{i:i\ne j} O_{ij}$ for $\p=1$.
\end{proof}

\begin{proof}[Proof of Theorem \ref{lpchar}]
The proof of Theorem \ref{lpchar} then follows the same
lines as Theorem \ref{l2char}.
For $\p>1$, we reduce as before to 
$Q \in \C\{x,y\}$ with $\deg_y Q < M$
and then by Proposition \ref{lpnorm},
$Q/P \in L^{\p}_{loc}$
if and only if for $j=1,\dots, M$
\[
\p \text{Ord}(Q(x,a_j(x;t)) + 2L_j - \p(\sum_{i=1}^{M} O_{ij}) > -1.
\]
Equivalently, 
\[
\text{Ord}(Q(x,a_j(x;t))) > \sum_{i=1}^{M} O_{ij} - \frac{2L_j+1}{\p}.
\]
An integer $k$ is greater than a real number $a$ if and only if $k \geq \lfloor a+1\rfloor$, so we
 can rewrite as 
\begin{equation} \label{rewrite}
\text{Ord}(Q(x,a_j(x;t))) \geq \sum_{i=1}^{M} O_{ij} - \left\lceil \frac{2L_j+1}{\p} \right\rceil +1.
\end{equation}

For $\p=1$, Proposition \ref{lpnorm} does
not apply unless $\deg_y Q < M-1$.  
If $\deg_y Q = M-1$ we can consider
\[
Q_1(x,y) = Q(x,y) - LT_Q(x) \frac{A(x,y)+tB(x,y)}{y-a_1(x;t)}
\]
where $LT_Q(x) \in \C\{x\}$ is the coefficient of $y^{M-1}$ in $Q(x,y)$.
Since $A+tB$ is monic in $y$, $(A+tB)/(y-a_1(x;t))$ is also monic in $y$
and therefore $\deg_y Q_1 < M-1$.
Applying Proposition \ref{lpnorm} to $Q_1$ we see that
$Q_1/P \in L^{1}_{loc}$ if and only if for $j=1,\dots, M$
\[
\text{Ord}(Q_1(x,a_j(x;t))) +2L_j - \sum_{i=1}^{M} O_{ij} = \text{Ord}(Q_1(x,a_j(x;t))) - \sum_{i:i\ne j} O_{ij}> -1.
\]
Here we use the fact that $\log(1/|x|) x^{s}$ and $\frac{x^s}{\log(1/|x|)}$ are
integrable around $0$ if and only if $s>-1$.
This is equivalent to
\[
\text{Ord}(Q_1(x,a_j(x;t))) \geq \sum_{i:i\ne j} O_{ij}.
\]

Now, notice that
\[
Q_1(x,a_1(x;t)) = Q(x,a_1(x;t)) - LT_Q(x)(A_y(x,a_1(x;t))+t B_y(x,a_1(x;t)))
\]
and since $\text{Ord}(A_y(x,a_1(x;t))+t B_y(x,a_1(x;t))) = \sum_{i:i\ne 1} O_{1i}$
we see that 
\[
\text{Ord}(Q_1(x,a_1(x;t))) \geq \sum_{i:i\ne 1} O_{i1}
\]
 if and only if
\[
\text{Ord}(Q(x,a_1(x;t))) \geq \sum_{i:i\ne 1} O_{i1}.
\]  
Also, for $j>1$, $Q_1(x,a_j(x;t)) = Q(x,a_j(x;t))$ so these evidently
have the same order of vanishing.
Finally, notice that $\frac{A+tB}{(y-a_1(x;t))P} \in L^1_{loc}$ by Lemma \ref{Iorder}
and therefore $Q/P \in L^1_{loc}$ if and only if $Q_1/P \in L^{1}_{loc}$
and therefore our necessary and sufficient condition for $Q/P\in L^{1}_{loc}$
is as claimed.
\end{proof}

Using Theorem \ref{lpchar} we can reprove the characterization of $\mcI^{\infty}_{P}$ 
 proven in \cite{BKPS} and Koll\'ar \cite{kollar}
 that we stated in the introduction:
  \Iinf*
 
 The implication 
 \[
 \prod_{j=1}^{M}(y+q_j(x), x^{2L_j}) \subset \mcI^{\infty}_{P}
 \]
 was proven in \cite{BKPS} while the opposite inclusion 
 is harder.  Some special cases were resolved
 in \cite{BKPS} but the full result was proven in Koll\'ar \cite{kollar};
 this is the part that we reprove here.  
 We can prove the following stronger claim.
 
 \begin{theorem} \label{Kmax} 
 Assume the setup of Theorem \ref{lpdim}.
 For $K =\max\{2L_1,\dots, 2L_M\}$
 we have 
 \[
 \mcI^{K+1}_{P} \subset \prod_{j=1}^{M}(y+q_j(x), x^{2L_j})
 \]
 and therefore
 \[
 \prod_{j=1}^{M}(y+q_j(x), x^{2L_j}) = \mcI^{K+1}_{P} = \mcI^{\infty}_{P}.
 \]
 \end{theorem}
 
\begin{proof}
As in the proof of Theorem \ref{lpchar},
we can reduce to $Q \in \C\{x,y\}$ with $\deg_y Q < M$.
For $\p = K+1$, we have that 
\[
\lceil (2L_j+1)/\p \rceil = 1 
\]
and referring to \eqref{rewrite}
if $Q/P \in L^{K+1}_{loc}$ then by Theorem \ref{lpchar} 
\begin{equation} \label{OrdQx}
\text{Ord}(Q(x,a_j(x;t))) \geq  \sum_{k=1}^{M} O_{kj} = 2L_j + \sum_{k:k\ne j} O_{kj}.
\end{equation}
Restating the formula \eqref{Qform}
\[
Q(x,y) = \sum_{j=1}^{M} \frac{Q(x,a_j(x;t))}{A_y(x,a_j(x;t)) + t B_y(x,a_j(x;t))} \frac{A(x,y)+tB(x,y)}{y- a_j(x;t)}
\]
we see that the expressions
\[
\frac{Q(x,a_j(x;t))}{A_y(x,a_j(x;t)) + t B_y(x,a_j(x;t))} 
\]
vanish to order at least $2L_j$ for $t$ proper by Corollary \ref{Bvanish} and \eqref{OrdQx}.
At the same time we have the formula
\[
 \frac{A(x,y)+tB(x,y)}{y- a_j(x;t)} = \prod_{k:k\ne j} (y+ q_k(x)+ x^{2L_k} \psi_k(x;t))
 \]
 which belongs to the ideal $\prod_{k: k\ne j}(y+q_k(x), x^{2L_k})$
 so that
 \[
 x^{2L_j}\frac{A(x,y)+tB(x,y)}{y- a_j(x;t)}
 \]
 belongs to the product ideal $\prod_{j=1}^{M} (y+q_j(x), x^{2L_j})$.
 This proves $Q$ belongs to the
 product ideal. 
 \end{proof}

\subsection{Final integrability characterization---Theorem \ref{finalint}}

Theorems \ref{l2char} and \ref{lpchar} both make use of
the order of vanishing of a polynomial or power series $Q(x,y)$ along a
branch $y = a_j(x;t)$ of $A+tB$.  Since these branches are somewhat
inaccessible it would be preferable to calculate a condition
only depending on the data $2L_1,\dots, 2L_M$ and $q_1(x),\dots, q_M(x)$
from $P$.  We are able to do just that in this section.

Recall the following definition.

\Onq*

We need the following simple fact.

\begin{lemma} \label{lemorddatum}
Let $Q(x,y) \in \C\{x\}[y]$ be nonzero.  Let $q(x) \in \R[x]$ with $q(0)=0$ and $L\geq 1$.
Let $r = O(2L,q, Q)$ and define
\[
g(x,y) = Q(x, -q(x) - x^{2L} y).
\]
There exists a finite set of real numbers $S$
such that for any $\psi\in \C\{x\}$ with $\psi(0) \notin S$
\[
g(x,\psi(x))
\]
vanishes to order $r$.
\end{lemma}

\begin{proof}
Note that $Q$ is identically zero if and only if $g$ is identically zero.
We can write $g$'s power series
\[
g(x,y) = \sum_{k \geq r} g_k(y) x^k
\]
for some $r \geq 0$ and $g_k(y) \in \C[y]$, where $g_r(y) \not\equiv 0$.
Then, $g(x,\psi(x))$ vanishes to order $r$ as long
as $g_r(\psi(0)) \ne 0$.  So, $S$ simply consists of the zeros of $g_r(y)$.
\end{proof}

We are finally in a place to prove Theorem \ref{finalint},
which characterizes integrability
in its most palatable form.
We will no longer need to refer to ``generic $t$
values'' and will obtain an essentially computable
condition.  Indeed, 
if we start with $P$ with no zeros in $\uhp^2$,
the data $2L_1,\dots, 2L_M$ and $q_1,\dots, q_M$
from from Theorem \ref{brackthm} can be
computed using Newton's method of rotating
rulers for computing Puiseux series expansions
as indicated by Koll\'ar \cite{kollar}.  Then,
the following theorem presents easily
computable conditions for a polynomial
or power series $Q$ to satisfy $Q/P \in L^{\p}_{loc}$
in terms of order of vanishing on a datum $(2L_j,q_j)$.

Let us recall

\finalint*

Recall $O_{ij}$ is given by Definition \ref{Oij}.

\begin{proof}
We consider the setup and conclusion of Theorem 
\ref{lpchar}.  In this setup, the branches $y-a_j(x;t)$
of $A+tB$ are arranged to be of the form 
\[
y+ q_j(x) + x^{2L_j} \psi_j(x;t)
\]
corresponding to the branches of $P$.

Given $Q\in \C\{x,y\}$, by the Weierstrass division
theorem we can write
\[
Q(x,y) = Q_0(x,y) + P(x,y) Q_1(x,y)
\]
where $Q_0 \in \C\{x\}[y]$, $\deg_yQ_0 < \deg_y P$,
and $Q_1(x,y) \in \C\{x,y\}$.
Since $P(x,y)$ satisfies $P/P=1 \in L^{\p}_{loc}$
for every $\p$, the order of vanishing conditions
on $P(x,y)Q_1(x,y)$ vacuously hold.  Therefore,
we can replace $Q$ with $Q_0$ and assume
$Q$ is a polynomial in $y$.

The order of vanishing of
\[
Q(x,-(q_j(x) +x^{2L_j} \psi_j(x;t)))
\]
equals $O(2L_j, q_j, Q)$
 as long as
$\psi_j(0;t)$ avoids a finite set as stated by 
Lemma \ref{lemorddatum}.
By Proposition \ref{distval}, for a fixed 
datum $(2L_j,q_j)$ the values of
$\psi_k(0;t)$ are all distinct as $t$ varies 
and as $k$ varies over indices that
correspond to a specific datum $(2L,q)$.
The upshot is that for all but finitely many $t \in \R$
the order of vanishing of
\[
Q(x,-(q_j(x) +x^{2L_j} \psi_j(x;t)))
\]
equals $O(2L_j, q_j, Q)$.
Using Theorem \ref{lpchar}, Theorem \ref{finalint} 
follows.
\end{proof}

\section{Derivative integrability} \label{Dersec}

To prove derivative integrability, we need a simple observation.

\begin{prop} \label{dersimp}
Let $Q(x,y) \in \C\{x,y\}$, $q(x) \in \C[x]$, $n \in \mathbb{N}$.
Recall Definition \ref{Onq}.
Then,
\[
n + O\left(n,q,\frac{\partial Q}{\partial y}\right) \geq O(n,q, Q).
\]
\end{prop}
\begin{proof}
We write
\[
G(x,y) = Q(x, x^n y - q(x)) = \sum_{k\geq N} g_k(y) x^k
\]
where $N = O(n,q,Q)$ and $g_N(y) \not\equiv 0$.
Then, 
\[
\frac{\partial G}{\partial y} = x^n \frac{\partial Q}{\partial y}(x,x^n y - q(x)) = \sum_{k\geq N} g_k'(y) x^k.
\]
If $g_N'(y) \not\equiv 0$, then $n + O(n,q,\frac{\partial Q}{\partial y}) = N$
otherwise $n + O(n,q,\frac{\partial Q}{\partial y}) > N$.
\end{proof} 

It is not so clear that when applying derivatives
we can simply replace a $P$ with no zeros
in $\uhp^2$ (or $\R\times \uhp$) with its local model $[P]$ from 
Theorem \ref{brackthm}.  Instead we need
to use Theorem \ref{Pfactored} to directly
check the following:

\begin{prop} \label{Pder}
Assume the setup and notation of Theorem \ref{Pfactored}.
Let $K = \max\{2L_1,\dots, 2L_M\}$.
Then, $\frac{\partial P}{\partial y}/ P \in L^{\p}_{loc}$ if and only if $\Im( \frac{\partial P}{\partial y}/ P) \in L^{\p}_{loc}$
if and only if $\p<1 + 1/K$.
\end{prop}

\begin{proof}
Referring to Theorem \ref{Pfactored} 
and using Proposition \ref{dersimp} we have 
\[
2L_j + O(2L_j, q_j, P_y) \geq O(2L_j, q_j, P)
\]
and since $P/P \in L^{\infty}_{loc}$
\[
O(2L_j, q_j, P) \geq \sum_{i=1}^{M} O_{ij}.
\]
Thus, 
\[
O(2L_j, q_j, P_y) \geq \sum_{i:i\ne j} O_{ij}
\]
implies by Theorem \ref{finalint} that
$P_y/P \in L^{\p}_{loc}$ if $2L_j< (2L_j+1)/\p $ for all $j=1,\dots, M$.
Thus, $P_y/P \in L^{\p}_{loc}$ for $\p < 1+1/K$.  
To prove this is sharp we need to dig a little deeper.

The $y$-logarithmic derivative of $P$ satisfies
\[
-\frac{P_y}{P} =- \frac{u_y}{u} + -\sum_{j=1}^{M} \frac{1}{y+q_j(x) + x^{2L_j} \psi_j(x^{1/k})}
\]
The imaginary part is 
\[
-\Im \frac{u_y}{u} + \sum_{j=1}^{M} \frac{x^{2L_j}\Im \psi_j(x^{1/k})}{|y+q_j(x) + x^{2L_j} \psi_j(x^{1/k})|^2}.
\]
Since $u$ is analytic and non-vanishing at $0$, we can disregard the term 
$-\Im\frac{u_y}{u}$.  
Since $\Im \psi_j(0) > 0$, 
we have 
\[
|y+q_j(x) + x^{2L_j} \psi_j(x^{1/k})| \approx |y+q_j(x) + i x^{2L_j}| 
\]
(i.e. the two sides are comparable in terms of constants for small $x$).
Thus, $\Im(P_y/P) \in L^{\p}_{loc}$ if and only if
\[
\sum_{j=1}^{M} \frac{x^{2L_j}}{|y+q_j(x) + i x^{2L_j}|^2} \in L^{\p}_{loc}.
\]
The term $\frac{x^{L_j}}{y+q_j(x) + i x^{2L_j}}$ 
belongs to $L^{2\p}_{loc}$
if and only if 
\[
L_j  \geq 2L_j - \left\lceil \frac{2L_j+1}{2\p}\right\rceil +1.
\]
This is an application of Theorem \ref{finalint} to the case of 
$y+q_j(x) + ix^{2L_j}$.
This inequality is equivalent to
\[
(2L_j +1)/(2\p) > L_j \text{ or simply } 1 + 1/(2L_j) > \p
\]
and thus $\Im(P_y/P) \in L^{\p}_{loc}$ if and only if
\[
1 + 1/K > \p
\]
for $K = \max\{2L_1,\dots, 2L_M\}$.
\end{proof} 

\begin{theorem} \label{Derthm}
Assume the setup and notation of Theorem \ref{Pfactored}.
Let $K = \max\{2L_1,\dots, 2L_M\}$.
If $Q(x,y) \in \C\{x,y\}$ and $Q/P \in L^{\infty}_{loc}$ then
\[
\frac{\partial}{\partial y} \frac{Q}{P} \in L^{\p}_{loc}
\]
for $\p < 1+1/K$.
\end{theorem}
\begin{proof}
By Theorem \ref{finalint}
\[
O(2L_j, q_j, Q) \geq \sum_{i=1}^{M} O_{ij}
\]
and by Proposition \ref{dersimp} 
\[
O(2L_j, q_j, Q_y) \geq \sum_{i:i\ne j} O_{ij}
\]
and therefore $Q_y/P \in L^{\p}_{loc}$ for $\p < 1+1/K$.
Since
\[
\frac{\partial}{\partial y} \frac{Q}{P} = \frac{Q_y}{P} - \frac{Q}{P} \frac{P_y}{P}
\]
Proposition \ref{Pder} yields the theorem.
\end{proof}

Setting $Q = \refl{P}$ yields part of a theorem from \cites{BPS18, BPS20}, that
the rational inner function $f = \refl{P}/P$ satisfies
$\frac{\partial f}{\partial y} \in L^{\p}_{loc}$ if and
only if $\p< 1+1/K$.  The ``only if'' now follows from Proposition \ref{Pder} since
\[
\frac{\partial}{\partial y} \frac{\refl{P}}{P} = \frac{\refl{P}}{P}\left( \frac{\refl{P}_y}{\refl{P}} - \frac{P_y}{P}\right) = 
(-2i)\frac{\refl{P}}{P} \Im \frac{P_y}{P}.
\]

\section{A basis for $\C\{x,y\}/(P,\bar{P})$} \label{sec:basis}

In this section we construct a concrete basis for the quotient $\mathcal{Q} := \C\{x,y\}/(P,\bar{P})$
by copying arguments from Fulton's \emph{Algebraic Curves} \cite{Fulton} Chapter 3.
For Proposition \ref{basis} below we do not need any
assumptions about properness of $t\in \R$.

\begin{restatable}{prop}{basis} \label{basis}
Assume $P(x,y) \in \C[x,y]$ 
has no zeros in $\R\times \uhp$ and
no factors in common with $\refl{P}$.
Let $A = (1/2)(P+\refl{P})$, $B = (1/(2i))(P-\refl{P})$.
For $j=1,\dots, M$, let $y-a_j(x;t)=0$ be the $M$ smooth
branches of $A+tB$.  
Let $m_j(t) = \text{Ord}(B(x,a_j(x,t)))$.
A basis for the quotient $\mathcal{Q} = \C\{x,y\}/(P,\bar{P})$ is given as follows.
Set
\[
F_k(x,y) : = \frac{A+tB}{\prod_{j=1}^{k} (y-a_j(x;t))}.
\]
(We suppress the dependence of $F_k$ on $t$ since
we do not need it.)
Then, the elements
\[
x^i F_k(x,y) \text{ where } 1\leq k\leq M,\ 0\leq i < m_k(t)
\]
are representatives for a basis of $\mathcal{Q}$.
\end{restatable}

Thus, a given element of $Q(x,y) \in \mathcal{Q}$ has a unique representative
\[
Q(x,y) = \sum_{k=1}^{M} c_k(x) F_k(x,y)
\]
where $\deg c_k(x) < m_k(t)$, and the dimension
of $\mathcal{Q}$ is given by $\sum_{j=1}^{M} m_j(t)$.

\begin{example}
Consider $P = (y+ix^2)(y+ix^4)$ (a small modification of Example \ref{exex}),
and corresponding
$A = (y-x^3)(y+x^3), B = x^2(1+x^2)y$.
The branches of $A$ are $y=\pm x^3$ and
\[
\text{Ord}(B(x,\pm x^3)) = 5
\]
so the dimension of $\mathcal{Q}$ is 10.
If we use a proper value of $t$, then
$A+tB$ has branches of the form $y + O(x^2) = 0, y + O(x^4)=0$.
In this case
\[
\text{Ord}(B(x, O(x^2))) = 4, \text{Ord}(B(x,O(x^4))) = 6
\]
and the sum is again $10$ as expected. $\diamond$
\end{example}

Proposition \ref{basis} is covering
well-trodden territory and so we only aim to give a highly specialized
proof for the benefit of the reader unfamiliar with this background.
First, let $R := \C\{x,y\}$ and for $G,H \in R$ we let $(G,H)$ denote
the ideal generated by $G,H$.

\begin{lemma} \label{basislemma}
Suppose $a(x) \in \C\{x\}$, $G_1(x,y) \in R$.
Let $G_0(x,y) = (y-a(x))G_1(x,y)$
and assume $H(x,y) \in R$.
Assume $G_0$ and $H$ have no common factors in $R$.
Assume (for induction purposes later)
that the quotient $R/(G_1,H)$ is finite dimensional.
Let 
\[
M_{G_1} : R/(y-a(x), H) \to R/(G_0, H)
\]
 denote
multiplication by $G_1$ and let
$\iota : R/(G_0, H) \to R/(G_1, H)$ denote
natural inclusion.
Then, 
\[
0 \to R/(y-a(x), H) \overset{M_{G_1}}{\to} R/(G_0, H) \overset{\iota}{\to} R/(G_1,H) \to 0
\]
is a short exact sequence.
A basis for $R/(y-a(x), H)$ consists of $x^i$ for $0\leq i < \text{Ord}(H(x,a(x)))$.
A basis for $R/(G_0,H)$
is given by
\[
x^iG_1(x,y) \text{ for } 0\leq i < \text{Ord}(H(x,a(x)))
\]
combined with representatives for a basis of $R/(G_1,H)$ pulled back to $R/(G_0,H)$.
In particular,
\[
\dim R/(G_0,H) = \text{Ord}(H(x,a(x))) + \dim R/(G_1,H).
\]
\end{lemma}

\begin{proof}
The multiplication map is well-defined
because elements of the ideal $(y-a(x),H)$, namely
\[
f_1(x,y)(y-a(x)) + f_2(x,y) H(x,y) \text{ for } f_1,f_2 \in R,
\]
 map
 to 
 \[
 f_1(x,y) G_0(x,y) + f_2(x,y)G_1(x,y)H(x,y)
 \]
 which belongs to the ideal $(G_0, H)$ in the target.
 The multiplication map is injective because
 if $f_0(x,y) \in R$ maps to the kernel, namely
 \[
 f_0(x,y)G_1(x,y) = g_1(x,y) G_0(x,y) + g_2(x,y) H(x,y)
 \]
 for some $g_1,g_2 \in R$
and then 
 \[
 (f_0(x,y)-(y-a(x))g_1(x,y))G_1(x,y) = g_2(x,y) H(x,y)
 \]
Therefore $G_1$ divides $g_2$ since $G_0$ and $H$ have
no common factors.
Writing $g_2 = G_1 g_3$ we see that
\[
f_0(x,y) = (y-a(x))g_1(x,y) + g_3(x,y) H(x,y)
\]
showing $f_0$ belongs to the ideal $(y-a(x), H)$,
i.e. equals $0$ in the quotient.

Next, the natural inclusion $\iota$
is well-defined because the ideals satisfy $(G_0,H) \subset (G_1,H)$.
This inclusion is surjective essentially because it is well-defined;
we can take any $f_0 \in R$ viewed as a representative of a
coset of $R/(G_0,H)$ and map it to its representative within $R/(G_1,H)$.

Finally, to show the range of $M_{G_1}$ equals the kernel of $\iota$,
it is clear that the range is contained in the kernel
as a multiple of $G_1$ belongs to the ideal $(G_1,H)$.
Given an element $f_1(x,y) G_0(x,y) + f_2(x,y) H(x,y)$
of the kernel of $\iota$, we have
that $M_{G_1}([f_1(x,y)(y-a(x))])$ maps to it
because
\[
G_1(x,y)(f_1(x,y)(y-a(x))) - (f_1(x,y) G_0(x,y) + f_2(x,y) H(x,y)) 
= -f_2(x,y) H(x,y) \in (G_0,H).
\]
This proves the given sequence is exact.

Now, the quotient $R/(y-a(x), H)$ has basis given by 
(representatives) of $x^i$ for $0\leq i< \text{Ord}(H(x,a(x)))$.
To see this, note that any $f(x,y) \in \C\{x,y\}$
can be written as $f(x,y) = f_0(x) + (y-a(x)) f_1(x,y)$
for $f_0(x) \in \C\{x\}$ and $f_1(x,y) \in R$.
(Simply consider that $f(x,y+a(x))$ can be written as
a term with only $x$'s and a multiple of $y$.) 
Note that $f_0(x) = f(x,a(x))$.
In particular, 
\begin{equation} \label{Hax}
H(x,y) = H(x,a(x)) + (y-a(x)) h(x,y)
\end{equation}
for some $h \in R$.
If $k = \text{Ord}(H(x,a(x)))$, we can write $H(x,a(x)) = x^k/h_0(x)$
where $h_0(0)\ne 0$.
Writing $f_0(x) = f_{00}(x) + x^k f_2(x)$ for $f_{00}(x) \in \C[x]$ 
with degree less than $k$,
we see that in $R/(y-a(x), H)$
\[
f(x,y) \equiv f_0(x) \equiv f_{00}(x) + h_0(x) H(x,a(x)) \equiv f_{00}(x).
\]
This shows $\{x^i: 0\leq i < k\}$ spans $R/(y-a(x), H)$.
The set is linearly independent because
if a polynomial $g(x) \in \C[x]$ of degree less than $k$
belongs to $(y-a(x), H)$ then 
\[
g(x) = g_1(x,y) (y-a(x)) + g_2(x,y) H(x,y).
\]
Setting $y=a(x)$ gives $g(x) = g_2(x,y)H(x,a(x))$
which vanishes to order higher than the degree of $g$.

Let $\{f_1(x,y), \dots , f_m(x,y)\}$ be
representatives for a basis for $R/(G_1,H)$.
Then, since the sequence in the lemma is exact
the following is a basis for $R/(G_0,H)$
\[
\{ x^iG_1(x,y): 0\leq i < k\} \cup \{f_1(x,y),\dots, f_m(x,y)\}.
\]

\end{proof}

Lemma \ref{basislemma} can be iterated to prove Proposition \ref{basis}.

\begin{proof}[Proof of Proposition \ref{basis}]
Note that the ideals $(A,B) = (A+tB,B)$ and $(P,\refl{P}) = (A+iB, A-iB)$
are equal since $A,B$ and $P,\refl{P}$ are linear combinations of each other.
By Theorem \ref{facerealbranches} we can factor
\[
A+tB = u(x,y;t)\prod_{j=1}^{M}(y-a_j(x;t))
\]
where for fixed $t$, $u(x,y;t) \in \C\{x,y\}$
is a unit; i.e. $u(0,0)\ne 0$.  
For our purposes we can divide out $u(x,y;t)$
and assume $A+tB = \prod_{j=1}^{M}(y-a_j(x;t))$
since the unit will not affect any basis properties.

If we set $F_0(x,y) = A+tB = (y-a_1(x;t))F_1(x,y)$, 
then we can apply Lemma \ref{basislemma}
to see that a basis for $R/(F_0,B)$
is given by
\[
x^i F_1(x,y) \text{ for } 0\leq i < \text{Ord}B(x,a_1(x;t)) = m_1(t)
\]
combined with a basis for $R/(F_1,B)$.
Since $F_1 = (y-a_2(x;t)) F_2$ we can iterate
to see that a basis for $R/(F_1,B)$
is given by
\[
x^i F_2(x,y) \text{ for } 0 \leq i < m_2(t).
\]
combined with a basis for $R/(F_2,B)$.
We can continue in this way all the way up to 
$R/(F_{M-1}, B)$
which has basis
\[
x^i F_M(x,y) \text{ for } 0 \leq i < m_M(t).
\]
Note here though that $F_M \equiv 1$.
\end{proof}

Now if we take $t \in \R$ to be proper as
in Theorem \ref{matchbranches} and order
the branches of $A+tB$ as in that theorem, Corollary \ref{Bvanish}
computes the quantities $m_j(t)$ in Proposition \ref{basis}
to be 
\[
m_j(t) = \sum_{i=1}^{M} O_{ij} = 2L_j + \sum_{i:i\ne j} O_{ij}.
\]
We have recovered the fact that the dimension of $\C\{x,y\}/(P,\refl{P})$
is given by
\[
\sum_{j=1}^{M} 2L_j + 2 \sum_{(i,j): i<j} O_{ij}
\]
which we stated in Section \ref{sec:locface}.

\section{Characterization of integrability quotients}

To achieve the main goals of the paper we need to combine
the explicit basis of Section \ref{sec:basis} and the parametrized 
orthogonal decomposition obtained from Theorem \ref{parseval}.

\begin{restatable}{theorem}{uberthm} \label{uberthm}
Let 
\[
P(x,y) = \prod_{j=1}^{M} (y+q_j(x)+ i x^{2L_j})
\]
where $L_1,\dots, L_M \in \mathbb{N}$ and $q_1,\dots, q_M \in \R[x]$ with 
$\deg q_j < 2L_j$ and $q_j(0)=0$.
Let $A = (1/2)(P+\refl{P}), B= (1/(2i))(P-\refl{P})$.
We can factor $A+tB = \prod_{j}(y-a_j(x;t))$
and for proper $t$ we can order so that $-a_j(x;t) = q_j(x) + x^{2L_j}\psi_j(x;t)$.
Define $F_k(x,y)$ as in Proposition \ref{basis} for $k=1,\dots, M$.

There exists a relabeling of the branches 
of $P$ so that the following holds.
Given $Q(x,y) \in \C\{x,y\}/(P,\refl{P})$ written in terms of
the basis given in Proposition \ref{basis}
\[
Q(x,y) = \sum_{k=1}^{M} c_k(x) F_k(x,y)
\]
where $c_k(x) \in \C[x]$ with $\deg c_k < 2L_k + \sum_{j: j\ne k} O_{jk}$,
we have that $Q/P \in L^{\p}_{loc}$
 if and only if for each $k=1,\dots,M$,
 $c_k(x)$ vanishes to order at least
 \[
 \lfloor (2L_k+1)(1-1/\p)\rfloor + \sum_{j: j<k} O_{jk}.
 \]
 \end{restatable}
 
 Before we get into the proof, let us detail an application.  
Theorem \ref{lpdim} from the introduction, which we restate here, is a consequence.

\lpdim*

\begin{proof}[Proof of Theorem \ref{lpdim}]
We simply count the number of 
free coefficients in the $c_k(x)$ polynomials in Theorem
\ref{uberthm}.
The dimension is
\[
\begin{aligned}
&\sum_{k=1}^{M} (2L_k + \sum_{j:j\ne k} O_{jk} - \lfloor (2L_k+1)(1-1/\p)\rfloor - \sum_{j:j<k} O_{jk})\\
&=
\sum_{k=1}^{M} (2L_k- \lfloor (2L_k+1)(1-1/\p)\rfloor)
+ \sum_{(j,k): j>k} O_{jk}\\
&=
\sum_{(j,k):j<k} O_{jk} 
+ \sum_{k=1}^{M} \left(\left\lceil \frac{2L_k + 1}{\p}\right\rceil -1\right).
\end{aligned}
\]

The formula for the dimension of $\mcI^{\infty}_P/(P,\refl{P})$ follows from Theorem \ref{Kmax}.
\end{proof}

To start proving Theorem \ref{uberthm} 
we obtain a technical characterization of integrability
in terms of the $c_k$ polynomials.

\begin{prop} \label{ccond}
Assume the setup of Theorem \ref{uberthm}.
Assume $t\in \R$ is proper.
Then, $Q/P \in L^{\p}_{loc}$ if and only if
for every $j=1,\dots, M$
\[
\text{Ord}\left(\sum_{k=j}^{M} c_k(x) \frac{1}{\prod_{\overset{i\ne j}{i\leq k}} (a_j(x;t)-a_i(x;t))}\right) 
\geq 
\lfloor (2L_j+1)(1-1/\p)\rfloor.
\]
\end{prop}

The condition in Proposition \ref{ccond} is obviously 
complicated.  Notice that for $j=M$ the sum has
only one term and we get the condition
\[
\text{Ord}\left(\frac{c_M(x)}{\prod_{i<M} (a_j(x;t) - a_i(x;t))} \right) \geq 
\lfloor (2L_M+1)(1-1/\p)\rfloor
\]
which simplifies to 
\[
\text{Ord}(c_M(x)) \geq 
\lfloor (2L_M+1)(1-1/\p)\rfloor 
+ \sum_{i: i<M} O_{iM}.
\]
As $j$ goes down from $M$ the expressions have more and more
terms. We will see in the next two subsections how to extract a simpler 
condition. 

\begin{proof}[Proof of Proposition \ref{ccond}]
We first apply Theorem \ref{parseval} to the choice $Q = F_k$;
we need to calculate
\[
\frac{F_k(x,a_j(x;t))}{A_y(x,a_j(x;t)) + t B_y(x,a_j(x;t))}.
\]
Note that
\[
F_k(x,a_j(x;t)) = 
\begin{cases} 0 & \text{ for } j > k \\
\prod_{i>k} (a_j(x;t) - a_i(x;t)) & \text{ for } j \leq k
\end{cases}
\]
and 
\[
A_y(x,a_j(x;t)) + t B_y(x,a_j(x;t)) = \prod_{i:i\ne j} (a_j(x;t) - a_i(x;t))
\]
so that for $j \leq k$
\[
\frac{F_k(x,a_j(x;t))}{A_y(x,a_j(x;t)) + t B_y(x,a_j(x;t))} = \frac{1}{\prod_{i\leq k, i\ne j} (a_j(x;t) - a_i(x;t))}
\]
and we get $0$ for $j > k$.

Applying Theorem \ref{parseval} we have
\begin{equation}\label{Fk}
F_k(x,y) = \sum_{j=1}^{k} \frac{1}{\prod_{\overset{i\ne j}{i\leq k}} (a_j(x;t)-a_i(x;t))} \frac{A+tB}{y-a_j(x;t)}.
\end{equation}

Now let us write $Q(x,y) \in \C\{x,y\}/(P,\refl{P})$
as in the statement of Theorem \ref{uberthm} using $F_k$.
Using \eqref{Fk} we have
\[
Q(x,y) = \sum_{j=1}^{M} \left(\sum_{k=j}^{M} c_k(x) \frac{1}{\prod_{\overset{i\ne j}{i\leq k}} (a_j(x;t)-a_i(x;t))}\right) \frac{A+tB}{y-a_j(x;t)}
\]
and 
\[
Q(x,a_j(x;t)) =
 \left(\sum_{k=j}^{M} c_k(x) \frac{1}{\prod_{\overset{i\ne j}{i\leq k}} (a_j(x;t)-a_i(x;t))}\right) \prod_{i:i\ne j}(a_j(x;t) - a_i(x;t)) 
 \]
 
By Theorem \ref{lpchar}, $Q/P \in L^{\p}_{loc}$ if and only if
\[
\text{Ord}(Q(x,a_j(x;t)))  - \sum_{i=1}^{M} O_{i,j} \geq 
1- \left\lceil \frac{2L_j+1}{\p} \right\rceil
\]
while by Theorem \ref{matchbranches}, $\text{Ord}(a_j(x;t) - a_i(x;t)) = O_{ij}$ for $t$ proper.
Putting this together we get a condition on the $c_k$ that
\[
\text{Ord}\left(\sum_{k=j}^{M} c_k(x) \frac{1}{\prod_{\overset{i\ne j}{i\leq k}} (a_j(x;t)-a_i(x;t))}\right) 
\geq 
2L_j + 1- \left\lceil \frac{2L_j+1}{\p} \right\rceil
=
\lfloor (2L_j+1)(1-1/\p)\rfloor.
\]
\end{proof}

The condition on the $c_k(x)$ from Proposition \ref{ccond} 
can be simplified to the following.

\begin{prop} \label{simplec}
Assuming the setup and conclusion of Proposition \ref{ccond}, 
it is possible to relabel the indices on the branches of $P$
(hence relabel the indices of the $c_k$) so that
the condition on the polynomials $c_k(x)$ is equivalent to:
$c_k(x)$ vanishes at least to the order 
\[
L_k(\p) + \sum_{j: j<k} O_{kj}
\]
for $k=1,\dots, M$
where 
\begin{equation} \label{Ljp}
L_j(\p) = 
\lfloor (2L_j+1)(1-1/\p)\rfloor
\end{equation}
\end{prop}

Theorem \ref{uberthm} is an immediate consequence.  
We do not see any way to avoid relabelling the $c_k$
or, what amounts to the same thing, relabelling the
branches of $P$.  The relabelling actually depends on $\p$
as shown in Section \ref{pexample} below.
We first prove the case $M=3$ to illustrate, because
the general proof is inductive and somewhat opaque.

\subsection{Working out the multiplicity three case}
Assume $M=3$ in Proposition \ref{simplec}.
We can relabel the three branches of $P$ 
in order to arrange that
\[
L_3(\p) + O_{31} + O_{32} \geq L_2(\p) + O_{21} + O_{23} ,\  L_1(\p) +O_{12} + O_{13}
\]
and then relabel $1,2$ potentially to arrange
\[
L_2(\p) \geq L_1(\p).
\]
Checking our order of vanishing conditions from Proposition \ref{ccond} with $j=M=3$, we first see 
\[
\frac{c_3(x)}{(a_3(x;t) - a_1(x;t))(a_3(x;t) - a_2(x;t))}
\]
vanishes to order at least $L_3(\p)$.
Thus, $c_3(x)$ vanishes to order at least $L_3(\p)+O_{13}+O_{23}$, again using
\[
\text{Ord}(a_i(x;t)- a_j(x;t)) = O_{ij}
\]
for $t$ proper.

For the $j=2$ term in Proposition \ref{ccond} we see
\[
\frac{c_2(x)}{a_2(x;t)-a_1(x;t)} + \frac{c_3(x)}{(a_2(x;t)-a_1(x;t))(a_2(x;t)-a_3(x;t))}
\]
vanishes to order at least $L_2(\p)$.
The second term above vanishes to order 
\[
\begin{aligned}
\text{Ord}(c_3) - O_{21} - O_{23} &\geq L_3(\p) + O_{13} + O_{23} - O_{21} - O_{23} \\
&\geq L_2(\p)+O_{21}+O_{23} - O_{21}-O_{23} \\
&=
L_2(\p)
\end{aligned}
\]
and in order for the sum to vanish to at least this order we must have the first term
\[
\frac{c_2(x)}{a_2(x;t)-a_1(x;t)},
\]
which vanishes to order $\text{Ord}(c_2)- O_{21}$,
vanishing to at least order $L_2(\p)$.  Namely,
$c_2(x)$ vanishes to order at least $L_2(\p) + O_{21}$.
Finally, we must have
\[
c_1(x) + \frac{c_2(x)}{a_1(x;t)-a_2(x;t)} + \frac{c_3(x)}{(a_1(x;t) - a_2(x;t))(a_1(x;t)-a_3(x;t))}
\]
vanishing to order at least $L_1(\p)$.
The third term vanishes to order
\[
\begin{aligned}
\text{Ord}(c_3) - O_{12} - O_{13} &\geq L_3(\p)+O_{31}+O_{32} - O_{12} - O_{13}\\
& \geq L_1(\p)+O_{12}+O_{13} - O_{12}-O_{13} \\
&\geq L_1(\p)
\end{aligned}
\]
and the second term vanishes to order
\[
\text{Ord}(c_2) - O_{12} \geq L_2(\p)+O_{21} - O_{12} \geq L_1(\p)
\]
and therefore $c_1(x)$ vanishes to order $L_1(\p)$.

\subsection{Multiplicity three example} \label{pexample}

The following example shows that our
procedure of ordering branches in the previous section
depends on $\p$.

\begin{example}
Consider $P = (y+ix^2)(y+x+ix^8)(y+ix^4)$ with the branches ordered
as shown.  
Note that
\[
2L_1= 2,\ 2L_2 = 8,\ 2L_3 = 4
\]
and
\[
O_{12} = 1,\ O_{13} = 2,\ O_{23} = 1.
\]
We compute
\[
L_1(\p) + O_{12} + O_{13} = \lfloor(3)(1-1/\p)\rfloor +3
\]
\[
L_2(\p)  + O_{21} + O_{23} = \lfloor (9)(1-1/\p)\rfloor+ 2
\]
\[
L_3(\p) + O_{31}+O_{32} = \lfloor(5)(1-1/\p)\rfloor+ 3.
\]
For $\p = 5/4$, we obtain $3, 3, 4$ and so the three branches
are in a valid ordering for the previous procedure.
On the other hand, for $\p=3$ we obtain the three values $5, 8, 6$
and the branches are not in a valid ordering for the procedure
in the previous section. 
Again for $\p=5/4$, we have $L_2(5/4) = 1 > L_1(5/4) = 0$ and the
two remaining branches are also in the correct ordering.
For $\p = 3$, we have $L_3(3) = 3 > L_1(3) = 2$ so the proper
ordering of branches for $\p=3$ is $(y+ix^2), (y+ix^4), (y+x+ix^8)$.

Let us find bases for $\mcI^{5/4}_{P}/(P,\refl{P})$ and $\mcI^{3}_{P}/(P,\refl{P})$.

Given a generic factoring $A+tB = (y-a_1(x;t))(y-a_2(x;t))(y-a_3(x;t))$,
meaning 
\[
a_1(x;t) = \alpha_1 x^2 + O(x^3)\ (\text{for } \alpha_1 \ne 0),\ a_2(x;t) = -x + O(x^6),\ a_3(x;t) = O(x^4)
\]
we have
\[
F_1(x,y) = (y-a_2(x;t)(y-a_3(x;t)),\ F_2(x,y) = (y-a_3(x;t)),\ F_3(x,y) = 1
\]
and basis for $\C\{x,y\}/(P,\refl{P})$ given by
\[
x^i F_1 \text{ where } 0\leq i < 5,\ x^i F_2 \text{ where } 0 \leq i< 10,\ x^i F_3 \text{ where } 0\leq i<7.
\]
By the previous section, $Q/P \in L^{5/4}_{loc}$
if and only if $Q$ has a representative in the quotient $\C\{x,y\}/(P,\refl{P})$
of the form
\[
c_1(x) F_1 + c_2(x) F_2 + c_3(x) F_3
\]
where $c_1$ vanishes to order at least $L_1(5/4) = 0$ and has degree less than 5,
$c_2$ vanishes to order at least $L_2(5/4) + O_{21} = 3$ and has degree less than 10,
$c_3$ vanishes to order at least $L_3(5/4) + O_{31} + O_{32} = 4$ and has degree less than 7.
Together we obtain the dimension
\[
5 + (10-3) + (7-4) = 15
\]
for $\mcI^{5/4}_{P}/(P,\refl{P})$.

On the other hand, for $\p = 3$ we cannot use the given ordering of branches.
So, instead we work with
\[
G_1 = (y-a_2(x;t))(y-a_3(x;t)),\ G_2 = (y-a_2(x;t)),\ G_3 = 1
\]
and obtain the basis
\[
x^i G_1 \text{ where } 0\leq i < 5,\ x^i G_2 \text{ where } 0 \leq i< 7,\ x^i G_3 \text{ where } 0\leq i<10.
\]
for $\C\{x,y\}/(P,\refl{P})$.  
By the previous section, $Q/P \in L^3_{loc}$ if and only if 
$Q$ has a representative in the quotient $\C\{x,y\}/(P,\refl{P})$
of the form
\[
c_1(x) G_1 + c_2(x) G_2 + c_3(x) G_3
\]
where $c_1$ vanishes to order at least $L_1(3) = 2$ and has degree less than 5,
$c_2$ vanishes to order at least $L_3(3) + O_{13} = 5$ and has degree less than 7,
$c_3$ vanishes to order at least $L_2(3) + O_{23} + O_{21} = 8$ and has degree less than 10.
Together we obtain the dimension
\[
(5-2) + (7-5) + (10-8) = 7
\]
for $\mcI^{3}_{P}/(P,\refl{P})$. 

To illustrate Theorem \ref{Kmax} we note that for $K = \max\{2L_1,2L_2,2L_3\} = 8$
we have $\mcI^{K+1}_{P} = \mcI^{9}_{P} = \mcI^{\infty}_{P}$.  
The dimension of $\mcI^{\infty}_{P}/(P,\refl{P})$ is $O_{12}+O_{13}+O_{23} = 4$
according to Theorem \ref{lpdim}.
The ordering of branches as in the case $\p=3$ is the correct ordering
for $\p\geq 9$ and we obtain that $Q/P \in L^{\infty}_{loc}$ if and only
if $Q$ has a representative in the quotient $\C\{x,y\}/(P,\refl{P})$ of
the form
\[
c_1(x) G_1 + c_2(x) G_2 + c_3(x) G_3
\]
where $c_1$ vanishes to order at least $L_1(9) = 2$ and has degree less than 5,
$c_2$ vanishes to order at least $L_3(9) + O_{13} = 6$ and has degree less than 7,
$c_3$ vanishes to order at least $L_2(9) + O_{23} + O_{21} = 10$ and has degree less than 10.
We obtain the dimension of $4$ as expected.  Note this implies $c_3(x) \equiv 0$ and the
three basis elements are simply
\[
x^2 G_1, x^3 G_1, x^4 G_1, x^6 G_2. \qquad \diamond
\]

\end{example}

\subsection{The general case}

For the general case, we relabel the indices of $1,\dots, M$ so that the following hold
\[
L_M(\p)  + \sum_{j:j\ne M} O_{M,j} \geq L_k(\p) + \sum_{j:j \ne k} O_{k,j} \text{ for } k \leq M
\]
and we recursively relabel so that more generally for $m\leq M$
\[
L_m(\p) + \sum_{j:j\ne m, j\leq m} O_{m,j} \geq L_k(\p) + \sum_{j:j\ne k, j \leq m} O_{k,j} \text{ for } k \leq m.
\]
Now the condition we are working with is that for $j=1,\dots, M$
\[
\sum_{k=j}^{M} c_k(x) \frac{1}{\prod_{i: i\ne j, i\leq k} (a_j(x;t)-a_i(x;t))}
\]
vanishes to order at least $L_j(\p)$.  

We now prove Proposition \ref{simplec} working backwards
from $k=M$. We know
\[
\frac{c_M(x)}{\prod_{i:i<M} (a_M(x;t)-a_i(x;t))}
\]
vanishes to order at least $L_M(\p)$ and therefore
$c_M(x)$ vanishes to order at least
\[
L_M(\p) + \sum_{i:i<M} O_{M,i}.
\]
Next, assume $c_{k}(x)$ vanishes to order at least $L_k(\p)+ \sum_{j:j<k} O_{kj}$
for $k= M, M-1,\dots, M-m$ where $m \geq 0$.
We will establish the claim for $k=M-m-1$. 
We know 
\begin{equation} \label{cjsum} 
\sum_{j=M-m-1}^{M} c_j(x) \frac{1}{\prod_{\overset{i\ne M-m-1}{i\leq j}} (a_{M-m-1}(x;t)-a_i(x;t))}
\end{equation}
vanishes to order at least $L_{M-m-1}(\p)$.  
For $j>M-m-1$ the term
\[
c_j(x) \frac{1}{\prod_{\overset{i\ne M-m-1}{i\leq j}} (a_{M-m-1}(x;t)-a_i(x;t))}
\]
vanishes to order at least
\[
\text{Ord}(c_j(x)) - \sum_{\overset{i\ne M-m-1}{i\leq j}} O_{i,M-m-1}
\geq L_j(\p) + \sum_{i:i<j} O_{ji} - \sum_{\overset{i\ne M-m-1}{i\leq j}} O_{i,M-m-1} \geq L_{M-m-1}(\p).
\]
Since \eqref{cjsum} vanishes to order at least $L_{M-m-1}(\p)$ 
by Proposition \ref{ccond} 
and all of the summands with $j>M-m-1$
vanish to at least this order, we must have 
\[
c_{M-m-1}(x) \frac{1}{\prod_{\overset{i\ne M-m-1}{i\leq M-m-1}} (a_{M-m-1}(x;t)-a_i(x;t))}
\]
vanishes to at least order $L_{M-m-1}(\p)$ and therefore $c_{M-m-1}(x)$ vanishes
to order at least
\[
L_{M-m-1}(\p) + \sum_{i< M-m-1} O_{M-m-1,i}
\]
as desired.
This proves Proposition \ref{simplec} and along
with it our main theorem, Theorem \ref{uberthm}.

\section{Appendix: Motivation from rational inner functions and sums of squares} \label{sec:sos}
This section discusses a specific problem
that led us to study the local $L^2$ integrability
question.  This section is purely here for motivation.

One natural motivation for the local $L^2$ theory for
rational functions comes from a circle of ideas around
polynomials with no zeros on the bidisk, associated sums of squares
decompositions, and associated determinantal representations.
This motivation is \emph{natural} because at the outset there is
no apparent connection to integrability yet it arises nonetheless.
The setup begins with a polynomial $p \in \C[z,w]$ with no zeros
on the bidisk $\D^2$.  The reflected polynomial here is
\[
\tilde{p}(z,w) = z^{n} w^{m} \overline{p(1/\bar{z}, 1/\bar{w})}
\]
where $(n,m)$ is the bidegree of $p$.  
We assume $p$ and $\tilde{p}$ have no common factors.
The function $\tilde{p}/p$ is rational, analytic on $\D^2$, 
has modulus less than $1$ in $\D^2$,
and has modulus $1$ on $\T^2$ outside of any boundary zeros.
The inequality $|\tilde{p}/p|\leq1$ on $\D^2$ can be certified via a special
sums of squares formula
\begin{align} \label{sos}
&|p(z,w)|^2 - |\tilde{p}(z,w)|^2 \\
&= (1-|z|^2)\sum_{j=1}^{n} |A_j(z,w)|^2 + (1-|w|^2)\sum_{j=1}^{m} |B_j(z,w)|^2
\end{align}
where $A_1,\dots, A_n, B_1,\dots, B_m \in \C[z,w]$. See \cite{gKpnozb}.
The sums of squares formula can be used to establish
a variety of useful formulas for $\tilde{p}/p$ (a transfer function 
realization formula) and $p$ (a determinantal representation).

To make a long story short, the polynomials in the sums of squares can be constructed
from Hilbert space 
operations by viewing polynomials as an inner product space with inner product
\[
\langle A, B\rangle = \int_{\T^2} \frac{A \bar{B}}{|p|^2} d\sigma \qquad (d\sigma = \text{ normalized Lebesgue measure})
\]
but for this to make sense we must have the constraint that $A/p, B/p \in L^2(\T^2)$.  
The story gets richer because the terms in \eqref{sos} need
not be unique even after taking into account unitary transformations applied
to the tuples $(A_1,\dots, A_n), (B_1,\dots, B_m)$.
The formula \eqref{sos} is minimal in the sense that the number of squares
required for each sum of squares term cannot be lowered from the 
bidegree $(n,m)$.
The sums of squares decomposition \emph{is} unique exactly when the following space
\[
\mathcal{K} = \{ A\in \C[z,w]: A/p\in L^2(\T^2), \deg_z A < n, \deg_w A < m\}
\]
is trivial.
There is an operator theoretic way to construct or determine all
of the minimal sum of squares decompositions. 
We can define a pair of commuting contractive matrices
on $\mathcal{K}$ via $T_1 = P_{\mathcal{K}} M_{z}|_{\mathcal{K}}$, $T_2^* = P_{\mathcal{K}} M^*_{w}|_{\mathcal{K}}$ where $M_z,M_w$ are the 
multiplication operators in $L^2(\frac{d\sigma}{|p|^2}, \T^2)$ and $P_{\mathcal{K}}$
is orthogonal projection onto $\mathcal{K}$ .  
It is non-trivial
to prove that $T_1, T_2^*$ commute.  
It is shown in \cite{gKintreg} (see also the more general work in \cite{BSV}) 
that
there is a 1-1 correspondence between joint invariant
subspaces of $(T_1,T_2^*)$ and minimal sums of squares
formulas for $p$.

Thus, it is of interest to understand $\mathcal{K}$ and closely related
spaces better.
In \cite{gKintreg}, the exact dimension was computed as
\[
\dim(\mathcal{K}) = nm - \frac{1}{2} N_{\T^2}(p,\tilde{p})
\]
where $N_{\T^2}(p,\tilde{p})$ is notation for the number of common 
zeros of $p$ and $\tilde{p}$ on $\T^2$ counted via intersection multiplicity.
In the course of determining this dimension, the larger space
\begin{equation} \label{largerspace}
\mathcal{K}_1 = \{ A\in \C[z,w]: A/p\in L^2(\T^2)\}
\end{equation}
was also studied and a concrete set of generators was determined
(essentially built from the sums of squares terms from earlier and
their reflected versions).
The proof surprisingly relied on the ``global'' B\'ezout's theorem 
adapted to $\mathbb{P}^1\times \mathbb{P}^1$
and we were able to avoid giving a direct local 
description of \eqref{largerspace}.  That problem
is remedied in the present paper.
All of this motivation is here to say that understanding $\mathcal{K}_1$
is interesting because of a problem that is not obviously
related.

\section{Frequently used theorems and definitions} \label{secfreq}

Here are the most frequently referenced theorems and definitions
gathered in one place.

\brackthm*

\IP*

\Onq*

\Oij*

\finalint*

\lpdim*

\facebranches*

\matchbranches*

\proper*

\Bvanish*

\Ijtp*

\onevarlp*

\ltwochar*

\lpchar*

\basis*

\uberthm*

\end{document}